\documentclass[maxcitenames=2]{fullpaper}

\pdfoutput=1

\title{Revisiting the directed pathwidth of \titlemath{h}-semicomplete digraphs}
\shorttitle{Revisiting the directed pathwidth of \titlemath{h}-semicomplete digraphs}

\author{
\authorname{Guilherme C. M. Gomes}
\affiliation{LIRMM, Université de Montpellier, CNRS, Montpellier, France}
\affiliation{Universidade Federal de Minas Gerais, Belo Horizonte, Brazil}
\authormail{gcm.gomes@dcc.ufmg.br}
\authorname{Raul Lopes}
\affiliation{Hamburg University of Technology, Institute for Algorithms and Complexity, Hamburg, Germany}
\authormail{raul.lopes@tuhh.de}
\authorname{Matthias Mnich}
\affiliation{Hamburg University of Technology, Institute for Algorithms and Complexity, Hamburg, Germany}
\authormail{matthias.mnich@tuhh.de}
}
\shortauthors{Gomes, Lopes, and Mnich}
\date{}

\usepackage{xspace}
\usepackage[disableredefinitions]{complexity}

\xspaceaddexceptions{[}

\usepackage{outlines}

\newlist{romanenumerate}{enumerate}{1}
\setlist[romanenumerate]{label=(\roman*), ref=(\roman*)}

\DeclareDelimFormat{namelabeldelim}{\addnbspace}
\zcLanguageSetup{english}{
  type = section,   Name-pl = sections,
  type = page,      Name-pl = pages,
  type = figure,    Name-pl = figures,
  type = table,     Name-pl = tables,
  type = item,      Name-pl = items,
  type = equation,  Name-pl = equations,
}
\NewDocumentCommand{\fref}{O{f} m}{\hyperref[#2]{\ensuremath{#1_{\ref*{#2}}}}}
\newcommand{\caseref}[2]{%
  \hyperref[#2]{\zcref*[noref]{#1}~\zcref*[noname]{#1}~\ref*{#2}}%
}
\providecommand{\titlemath}[1]{\texorpdfstring{$#1$}{#1}\xspace}

\newkeytheorem{corollary}[style=paperplain, sibling=theorem,
  refname={corollary,corollaries}, Refname={Corollary,corollaries}]
\newkeytheorem{observation}[style=paperplain, sibling=theorem,
  refname={observation,observations}, Refname={Observation,observations}]
\newkeytheorem{claim}[style=paperdefinition, sibling=theorem,
  refname={claim,claims}, Refname={Claim,claims}]

\newenvironment{cproof}
  {\begin{proof}[Proof of the claim]}
  {\end{proof}}

\usepackage{tikz}
\usetikzlibrary{positioning, arrows.meta, decorations.pathmorphing, calc, fit, backgrounds}

\tikzset{faded/.style={gray,very thin}}
\tikzset{vertex/.style={draw,circle,minimum size=5pt,inner sep=0pt}}
\tikzset{novertex/.style={circle,minimum size=5pt,inner sep=0pt}}
\tikzset{blackvertex/.style={draw,circle,minimum size=5pt,inner sep=0pt, fill=black}}
\tikzset{redvertex/.style={draw,circle,minimum size=5pt,inner sep=0pt, fill=red}}
\tikzset{bluevertex/.style={draw,circle,minimum size=5pt,inner sep=0pt, fill=blue}}
\tikzset{greenvertex/.style={draw,circle,minimum size=5pt,inner sep=0pt, fill=green}}
\tikzset{arrow/.style={-{Latex[scale=1]},shorten >= 0pt}}
\tikzset{snake it/.style={decorate, decoration=snake}}

\newcommand{\mathtype}[1]{\mathrm{#1}}

\DeclareMathOperator{\dpw}{\mathtype{dpw}}
\DeclareMathOperator{\dcw}{\mathtype{dcw}}

\DeclareMathOperator{\tower}{\mathtype{tower}}

\newcommand{\Ocal}{\ensuremath{\mathcal{O}}}
\newcommand{\Ramsey}{\ensuremath{\mathtype{R}}}
\newcommand{\ADP}{\textsc{ADP}\xspace}
\newcommand{\DDP}{\textsc{DDP}\xspace}
\newcommand{\detours}[3]{\ensuremath{P_{#1}(#2 \to #3)}}
\DeclareMathOperator{\join}{\bigoplus}

\usepackage{environ}
\usepackage{booktabs}
\usepackage{tabularx}
\newcounter{boxproblem}
\newlength{\problemboxsep}
\NewEnviron{boxproblem}[2][]{%
  \par\addvspace{\medskipamount}%
  \noindent
  \begingroup
    \refstepcounter{boxproblem}%
    \setlength{\extrarowheight}{2pt}%
    \begin{tabularx}{\textwidth}{@{}>{\bfseries}l@{\hspace{1.2em}}>{\raggedright\arraybackslash}X@{}}
      \toprule
      \multicolumn{2}{@{}l@{}}{%
        \scshape #2}\\
      \addlinespace[\problemboxsep]
      \midrule
      \addlinespace[\problemboxsep]
      \BODY\\
      \addlinespace[\problemboxsep]
      \bottomrule
    \end{tabularx}%
  \endgroup
  \par\addvspace{\medskipamount}%
}
\zcRefTypeSetup{boxproblem}{
  Name-sg = Problem, name-sg = problem,
  Name-pl = problems, name-pl = problems,
}

\newcommand{\hsemi}[1]{$#1$\nobreakdash-\nobreak\hspace{0pt}semicomplete\xspace}

\begin{document}
\pagenumbering{roman}
\maketitle

\begin{abstract}
  A digraph is \emph{\hsemi{h}} if every vertex has at most~$h$ non-neighbors.
  In semicomplete digraphs, \emph{$k$-triples} are obstructions to bounded directed pathwidth with many algorithmic applications: Fradkin and Seymour~[JCTB~2013] showed that a large triple contains a subdivision of every small digraph, and triples were used for irrelevant vertex arguments in routing problems by Fomin and Pilipczuk~[JCTB~2019] and Gomes, Lopes, and Sau~[ICALP~2025].
  Fradkin and Seymour also showed that there is a function~$f(k)$ such that every semicomplete digraph of directed pathwidth at least $f(k)$ contains a $k$-triple, and their proof gives $f(k)$ triply exponential in~$k$.
  For \hsemi{h} digraphs, a result of Kitsunai et~al.~[ESA~2015] also yields a $k$-triple, but only existentially, and finding one in fixed-parameter time remained open.

  We improve the bound on $f(k)$ to a single exponential and extend the result to \hsemi{h} digraphs: we show that every \hsemi{h} digraph of directed pathwidth greater than $(h+1) \cdot k \cdot 2^{27k+9}$ contains a $k$-triple, which we find in polynomial time, and no bound subexponential in~$k$ or sublinear in~$h$ is possible.
  We also extend the approximation algorithm for directed pathwidth of Fomin and Pilipczuk to \hsemi{h} digraphs, so its certificates of large directed pathwidth remain valid.
  Our proof is much shorter than the analysis of Kitsunai et~al., and while their algorithm is exact, ours runs in polynomial time.
  As an application, we show that \textsc{Arc-Disjoint Paths} on \hsemi{h} digraphs is fixed-parameter tractable parameterized by the number of paths and~$h$, generalizing the semicomplete case.
  Finally, unlike in semicomplete digraphs, directed pathwidth does not bound directed cliquewidth: for every~$k$, some acyclic \hsemi{1} digraph, so of directed pathwidth zero, has directed cliquewidth at least~$k$.
\end{abstract}

\metadatum{Keywords.}{Semicomplete digraphs, parameterized complexity, directed pathwidth, directed cliquewidth, triples}
\metadatum{Funding.}{\emph{Guilherme C. M. Gomes}: Funded by the European Union, project PACKENUM, grant number 101109317. Views and opinions expressed are however those of the author only and do not necessarily reflect those of the European Union. Neither the European Union nor the granting authority can be held responsible for them.\\
\emph{Raul Lopes}: HIDSS-0002 DASHH (Data Science in Hamburg - Helmholtz Graduate School for the Structure of Matter).}
\thispagestyle{empty}
\clearpage
\pagenumbering{arabic}

\section{Introduction}
\label{sec:intro}
The \emph{treewidth} of undirected graphs~\cite{RobertsonS86tw}, which can be understood as an estimation of how tightly a graph can be approximated by a tree, is arguably the most successful example of a graph parameter from the viewpoint of algorithmic applications.
For instance, the celebrated theorem of \textcite{Courcelle90} states that every graph problem expressible in monadic second-order logic is fixed-parameter tractable parameterized by the treewidth of the input graph.
Structurally, the fundamental reasons why the treewidth of some graphs is large are well understood and often used to push applications of treewidth outside the realm of graphs of bounded treewidth.
For example, the canonical obstruction to bounded treewidth is given by the Grid Theorem by \textcite{RobertsonS86}, which states that there is a function~$f$ such that every graph with treewidth at least $f(k)$ contains a $(k \times k)$-grid as a minor.
In later work~\cite{RobertsonS95,RobertsonS12}, they use large grid minors, together with further structure, to find \emph{irrelevant vertices} for the \textsc{Disjoint Paths} problem, and thereby reduce the problem to graphs of bounded treewidth.
We discuss irrelevant vertices later in this section.

For digraphs, \textcite{JohnsonRST01} introduced the parameter known as \emph{directed treewidth}, but it has not seen the same success as undirected treewidth regarding algorithmic applications.
Although a fixed-parameter approximation algorithm for directed treewidth is known~\cite{CamposLMS22}, \textcite{LampisKM11} showed that \textsc{Directed Hamiltonian Cycle} is \W[2]-hard parameterized by directed treewidth and that \textsc{Max Directed Cut} is \NP-hard on directed acyclic graphs, which have directed treewidth zero, whereas their undirected versions, \textsc{Hamiltonian Cycle} and \textsc{Max Cut}, are fixed-parameter tractable parameterized by treewidth.
In fact, it is a topic of discussion whether there are \emph{any} good digraph width measures, as discussed by Ganian et al.~\mbox{\cite{GanianHKMORS16,GanianHKLOR14}}.
A classic example of a problem resistant to many approaches based on width measures of digraphs to design fixed-parameter algorithms is the \textsc{Arc-Disjoint Paths} (\textsc{ADP}) problem.
Given a digraph~$D$ and pairs~$(s_1, t_1), \ldots, (s_p, t_p)$ of vertices of~$D$, \textsc{ADP} asks whether~$D$ has pairwise arc-disjoint paths~$P_1, \ldots, P_p$ such that~$P_i$ connects a terminal~$s_i$ to terminal~$t_i$, for $i = 1,\hdots,p$ (see \zcref{sec:adp-dynamic-programming} for the formal definition).

\textcite{FortuneHW80} showed that \textsc{ADP} is \NP-complete even for two paths, and \textcite{Slivkins10} showed \W[1]-hardness parameterized by the number of paths even when restricted to directed acyclic graphs, which immediately implies that the \W[1]-hardness extends to digraphs of bounded directed treewidth.
Those hardness results raise the question of which parameters and classes of digraphs allow one to derive algorithmically useful properties.
This question has found particularly satisfactory answers in the class of \emph{semicomplete digraphs}.
A digraph~$D$ is \emph{semicomplete} if any pair of vertices of~$D$ is connected by at least one arc.
We are primarily concerned with directed pathwidth, and so refer the reader to \textcite{GurskiKRW21} for a more general discussion on other width parameters on semicomplete digraphs.

Directed pathwidth behaves very differently on semicomplete digraphs compared to general digraphs.
\textcite{FradkinS15} showed that, in semicomplete digraphs, \textsc{ADP} is fixed-parameter tractable parameterized by the number~$p$ of paths that we wish to find.
\textcite{FominP19} improved the dependence of the run time on the number~$n$ of vertices from~$n^5$ to~$n^3$, with an algorithm that also solves \textsc{Rooted Immersion}, a generalization of \textsc{ADP} that asks for an injective map from the vertices of a digraph~$H$ to the vertices of the input digraph, with some images prescribed, together with arc-disjoint paths joining the images of the ends of every arc of~$H$.
They used three ingredients to design a ``win-win'' strategy based on irrelevant vertices.
The first ingredient is a robust structural theory for directed pathwidth on semicomplete digraphs developed by \textcite{FradkinS13}; of particular interest in it is the existence of a simple obstruction, known as a \emph{$k$-triple} (see \zcref{def:k-triple}), which, for some function~$f$, every semicomplete digraph with directed pathwidth at least~$f(k)$ contains.
The second ingredient is an algorithm that, given a $g(p)$-triple that avoids the terminals, finds an \emph{irrelevant vertex} of the input instance in polynomial time, where a vertex is irrelevant if it can be safely removed from the input digraph without changing the answer to the problem.
The final ingredient is a dynamic programming algorithm parameterized by~$p$ and the directed pathwidth~$\dpw(D)$ of the input digraph~$D$.
Every semicomplete digraph of directed pathwidth~$k$ has directed cliquewidth at most~$k+2$~\cite[Lemma 2.14]{FominP19}, and \textsc{ADP} can be expressed in monadic second-order logic without quantification over arc sets~\cite{GanianHO11}.
Together with the meta-theorem of \textcite{CourcelleMR00} for this logic on graphs of bounded cliquewidth, which also holds for digraphs~\cite[Theorem 2.15]{FominP19}, this already implies that \textsc{ADP} is fixed-parameter tractable parameterized by $p + \dpw(D)$.
Hence the interest in the dynamic programming algorithm is mainly its (improved) run time.

Beyond their work on \textsc{ADP}, \textcite{FominP19} also contribute to the structural theory of directed pathwidth on semicomplete digraphs, especially through efficient algorithms to compute important elements of the theory.
In particular, they design a constant-factor polynomial-time approximation algorithm for directed pathwidth in semicomplete digraphs, as well as a fixed-parameter algorithm.
Concretely, when given a semicomplete digraph~$D$, the approximation algorithm obtains a \emph{degree tangle} or a \emph{matching tangle} (see \zcref{def:degree-tangles,def:matching-tangles}), or a \emph{path decomposition} (see \zcref{def:pathwidth}) of bounded width of~$D$.
Both tangles are shown to be obstructions to bounded directed pathwidth and then individually used to build \emph{jungles} (see \zcref{def:jungles}), the original obstruction of this theory~\cite{FradkinS13}.

Given these structural insights~\cite{FradkinS13} and algorithmic successes~\cite{FominP19} for the directed pathwidth of semicomplete digraphs, a natural question is to which larger classes of digraphs those can be extended.
This question motivates our current paper.

Our main contributions answer this question for \emph{\hsemi{h}} digraphs, which are the digraphs whose vertices each have at most~$h$ non-neighbors.
Thus, \emph{\hsemi{h}} digraphs are universal in the sense that any digraph is \emph{\hsemi{h}} for some value of~$h$.

\textcite{KitsunaiKT15} gave a fixed-parameter algorithm to compute the directed pathwidth of \hsemi{h} digraphs, generalizing the fixed-parameter algorithm for semicomplete digraphs~\cite{FominP19}.
They also show that obstructions for bounded directed pathwidth of semicomplete digraphs remain obstructions in \hsemi{h} digraphs.
Namely, they prove that there is a function~$f(k, h)$ such that every \hsemi{h} digraph of directed pathwidth at least $f(k, h)$ contains a semicomplete subgraph of directed pathwidth at least~$k$.
Building a semicomplete digraph from an \hsemi{h} one is easy.
Start with a set~$X$ of candidate vertices, initially containing all vertices of the given digraph~$D$, and an empty digraph~$D'$.
Then iteratively pick a vertex~$v \in X$ to include in~$D'$, and remove from~$X$ both~$v$ and all \emph{non-neighbors} of~$v$.
After selecting~$\ell$ vertices, at most $\ell \cdot (h+1)$ vertices have been removed from~$X$ and the selected vertices induce a semicomplete digraph.
The difficult part is guaranteeing that obstructions to directed pathwidth survive such a selection, as it may very well be the case that the deleted vertices contained most of a $k$-triple, for example.
Therefore, \textcite{KitsunaiKT15} work through a careful and difficult analysis, heavily reliant on probabilistic techniques, to show that some large obstruction survives a sampling similar to the selection procedure we just described.

Their work~\cite{KitsunaiKT15}, however, stops short of producing jungles in fixed-parameter time.
First, their proof produces, in a semicomplete subgraph, one of \emph{three} obstructions: a degree tangle, a matching tangle, or what they call a \emph{spider}.
Although it is reasonable to expect that jungles can be obtained from degree tangles and matching tangles in a way similar to the semicomplete case~\cite{FominP19}, it is not clear whether jungles can be constructed from spiders in polynomial time.
Second, their proof is existential and, as far as we could verify, there is no indication that it can be easily translated into a fixed-parameter algorithm.
A na\"{\i}ve \XP algorithm to find $k$-triples is immediate since a $k$-triple is formed by $3k$ vertices (see \zcref{def:k-triple}).

\subparagraph*{Our contributions.}
Our first result is a bound on the directed pathwidth, single exponential in~$k$, that forces a $k$-triple in an \hsemi{h} digraph, together with a polynomial-time algorithm that finds the triple.

\begin{theorem}[Simplified]
\label{thm:large-dpw-to-triples-simplified}
  Let~$k$ be a positive integer,~$h$ be a non-negative integer, and~$D$ be an \hsemi{h} digraph on~$n$ vertices.
  If $\dpw(D) > (h+1) \cdot k \cdot 2^{27k + 9}$, then~$D$ contains a $k$-triple, which can be found in $\min\{{(h+1)}^3 \cdot 2^{\Ocal(k)} \cdot n^2, \Ocal(n^5)\}$~time.
\end{theorem}

The lower bounds of \zcref{lem:large-dpw-to-triples-tight} show that no bound subexponential in~$k$ or sublinear in~$h$ is possible.
Thus, the linear dependence on~$h$ is best possible, and the dependence on~$k$ is best possible up to the constant in the exponent.
The gain is large even for semicomplete digraphs, where the proofs of \textcite[(1.1) and (2.6)]{FradkinS13} give a bound triply exponential in~$k$.
For \hsemi{h} digraphs, that bound applied to the semicomplete subgraph of large directed pathwidth given by \textcite[Theorem 2]{KitsunaiKT15} yields a $k$-triple, but only by an existential proof.
On the algorithmic side, \textcite[Section 6]{GomesLS25} asked for a fixed-parameter algorithm that computes a triple in \hsemi{h} digraphs, and ours runs in polynomial time.
For semicomplete digraphs, \textcite[Lemma 44]{Pilipczuk13} builds a $k$-triple in $\Ocal(n^3 \cdot \log n)$ time from a certificate of large directed pathwidth, namely a given jungle whose size is triply exponential in~$k$.
Such a jungle is found in polynomial time when the directed pathwidth is large enough~\cite[Theorem 4.12; Lemmas 3.9 and 3.12]{FominP19}, as also done in the same work~\cite[proof of Theorem 106]{Pilipczuk13}, so for semicomplete digraphs the gain of \zcref{thm:large-dpw-to-triples-simplified} is the single-exponential threshold.
As a consequence of \zcref{thm:large-dpw-to-triples-simplified}, the approach for \textsc{Rooted Immersion} in semicomplete digraphs~\cite[Theorem 9.9]{FominP19} gives a run time which is only doubly exponential in ${(|V(H)| + |E(H)|)}^2$, instead of quadruply exponential (see \zcref{sec:rooted-immersion}).

Our second result extends the approximation algorithm for directed pathwidth by \textcite[Theorem 4.12]{FominP19} from semicomplete to \hsemi{h} digraphs.

\clearpage
\pagebreak
\begin{theorem}[Simplified]
\label{thm:tangles-or-path-decomposition-simplified}
  Let~$D$ be an \hsemi{h} digraph on~$n$ vertices and~$k$ be a positive integer.
  There is an algorithm with run time~$\Ocal((k + h) \cdot n^2)$ that outputs a path decomposition of~$D$ of width at most $6(k + h)$ or a certificate that $\dpw(D) > k$.
\end{theorem}

The certificate is a degree tangle or a matching tangle of~$D$ itself (\zcref{def:degree-tangles} and \zcref{def:matching-tangles}), so these obstructions remain obstructions in \hsemi{h} digraphs.
In addition, \zcref{cor:dpw-approximation} approximates directed pathwidth within a factor~\mbox{$6(h + 1)$} \emph{in polynomial time}, which for $h = 0$ recovers the factor six for semicomplete digraphs~\mbox{\cite[Theorem 4.12]{FominP19}}.
In contrast, deciding whether $\dpw(D) \leq k$ via the algorithm by \textcite[Theorem 1]{KitsunaiKT15} requires exponential time \mbox{${(h + 2k + 1)}^{2k} \cdot n^{\Ocal(1)}$}.

Our third result extends the fixed-parameter algorithm by~\textcite[Theorem~9.9]{FominP19} for \ADP from semicomplete to \hsemi{h} digraphs, parameterized by~$k$ and~$h$.

\begin{theorem}[store=adpfpt]
    \label{thm:adp-fpt-simplified}
    Let $s = 2^{\Ocal(k^3 + kh)}$.
    There is an algorithm with run time~$2^{s} \cdot n^3$ that solves every instance of {\sc ADP} with~$k$ paths on an \hsemi{h} digraph with~$n$ vertices.
\end{theorem}

For the larger class of digraphs with no independent set of size~$h + 2$, which contains the \hsemi{h} digraphs, \textcite[(4.5)]{FradkinS15} give an \XP algorithm for \ADP parameterized by~$k$ and~$h$, whereas \zcref{thm:adp-fpt-simplified} gives a fixed-parameter algorithm parameterized by~$k$ and~$h$.

Our fourth result is a proof that directed pathwidth does not bound directed cliquewidth (\zcref{def:directed-cliquewidth}) in \hsemi{h} digraphs, already for $h = 1$.
This contrasts with the fact that in semicomplete digraphs, their directed cliquewidth is at most the directed pathwidth plus two~\cite[Lemma 2.14]{FominP19}.

\begin{theorem}[store=cliquewidthpathwidth]
    \label{thm:cliquewidth-and-pathwidth}
    For every integer~$k \geq 1$, there exists an acyclic \hsemi{1} digraph, hence of directed pathwidth~$0$, whose directed cliquewidth is at least~$k$.
\end{theorem}

\subparagraph*{Further research.}
Two questions arise naturally.
First, by \zcref{thm:large-dpw-to-triples-simplified} and \zcref{lem:large-dpw-to-triples-tight}, the exponent of the best possible bound lies between $(2/3 - o(1))k$ and $(27 + o(1))k$, and we do not know the right one, even for tournaments.
Second, we consider the \textsc{Directed Disjoint Paths} (\DDP) problem, which is defined as \ADP except that the paths must be pairwise vertex-disjoint.
On semicomplete digraphs, \DDP is fixed-parameter tractable parameterized by the number~$k$ of paths and the directed pathwidth, and the question is open parameterized by~$k$ alone~\cite[Section 11]{FominP19}, that is, for large directed pathwidth.
We remark that an example~\cite{FominP19} shows that irrelevant vertex arguments relying on $k$-triples are not viable for \DDP.
On \hsemi{h} digraphs, it is open parameterized by $k + h$ for both small and large directed pathwidth (compare with \textcite[Section 6]{GomesLS25}).

\subparagraph*{Organization.}
In \zcref{sec:overview} we explain our methods, and in \zcref{sec:prelim} we fix the notation and recall the known results that we use.
In \zcref{sec:polytime-app-algorithm} we prove the complete version of \zcref{thm:tangles-or-path-decomposition-simplified}, and in \zcref{sec:large-dpw-to-triples} the complete version of \zcref{thm:large-dpw-to-triples-simplified}.
In \zcref{sec:adp-algorithm} we introduce anti-triples and prove \zcref{thm:adp-fpt-simplified}.
Finally, in \zcref{sec:cliquewidth-and-pathwidth} we prove \zcref{thm:cliquewidth-and-pathwidth}.

\section{Overview of our techniques}
\label{sec:overview}

Our main result (\zcref{thm:large-dpw-to-triples-simplified}, in full \zcref{thm:large-dpw-to-triples}) is that directed pathwidth more than $f(k, h) = (h+1) \cdot k \cdot 2^{27k + 9}$ forces a $k$-triple in every \hsemi{h} digraph.
As for semicomplete digraphs~\cite[Lemma 44 and the proof of Theorem 106]{Pilipczuk13}, the $k$-triple can also be found in polynomial time.
For semicomplete digraphs, the proofs of \textcite[(1.1) and (2.6)]{FradkinS13} give a bound triply exponential in~$k$, since the second rests on Ramsey's theorem, whereas our bound~$f(k, h)$ is single exponential in~$k$ and linear in~$h$.
The new idea is to build the triple by counting, without Ramsey's theorem, and the larger class costs only a factor~$h + 1$.
This section explains how, and offers an overview of the techniques used in our main contributions.
The lower bound in~$k$ of \zcref{lem:large-dpw-to-triples-tight} is obtained through standard probabilistic arguments.
The digraphs of \zcref{thm:cliquewidth-and-pathwidth} are obtained by removing a matching from an acyclic tournament (\zcref{def:mirrored-grid}).

\subparagraph*{Approximating directed pathwidth.}
For semicomplete digraphs, the approximation algorithm of \textcite[Theorem 4.12]{FominP19} outputs a path decomposition of width at most $6k$ or a tangle that certifies $\dpw(D) > k$.
A \emph{degree tangle} is a set of $4k + 2$ vertices whose out-degrees differ by at most~$k$, and a \emph{matching tangle} consists of $k + 1$ disjoint arcs such that the out-degree of every head exceeds that of every tail by more than~$k$ (\zcref{def:degree-tangles,def:matching-tangles}).
Both certify large directed pathwidth because a path decomposition of small width sorts the vertices by out-degree, and the corresponding proofs~\cite[Lemmas 3.8 and 3.11]{FominP19} need semicompleteness only to count as out-neighbors of a vertex all the vertices that are not its in-neighbors.
In an \hsemi{h} digraph, every vertex has at most~$h$ non-neighbors, so these counts are off by at most~$h$.
Allowing a slack of~$h$ in the tangles, namely $2h$ more vertices and out-degrees within $k + h$ for degree tangles and a gap of more than $k + h$ for matching tangles, lets us repeat the arguments for semicomplete digraphs (\zcref{lem:degree-tangles-large-dpw,lem:matching-tangles-large-dpw}).
Hence their algorithm runs on~$D$ itself with $k + h$ in place of~$k$, without first making~$D$ semicomplete (\zcref{thm:tangles-or-path-decomposition}).
The triples of the next part, in contrast, do not come from these tangles of~$D$ but from a \emph{completion} of~$D$, a semicomplete digraph obtained from~$D$ by adding arcs between non-neighbors (\zcref{def:detours-and-completions}).

\subparagraph*{Finding triples by counting.}
A $k$-triple~$(A, B, C)$ consists of three disjoint sets of~$k$ vertices such that~$A$ is \emph{complete to}~$B$, that is, every vertex of~$A$ has an arc to every vertex of~$B$,~$B$ is complete to~$C$, and~$k$ disjoint arcs go from~$C$ to~$A$ (\zcref{def:k-triple}).
Every large enough jungle of a semicomplete digraph contains a $k$-triple, but the proof by \textcite[(2.6)]{FradkinS13} rests on Ramsey's theorem and needs a jungle triply exponential in~$k$ (see \zcref{sec:rooted-immersion}).
We replace the Ramsey dependence by counting arguments as follows.

First, we complete~$D$ by adding arcs between every two non-adjacent vertices.
For distinct vertices~$u$ and~$v$, we denote by $\detours{D}{u}{v}$ the set of vertices that are out-neighbors of~$u$ and in-neighbors of~$v$, that is, the middle vertices of the \emph{detours} from~$u$ to~$v$, the paths of length two.
For every two non-neighbors~$u$ and~$v$, we add the arc from~$u$ to~$v$ if $|\detours{D}{u}{v}| \geq |\detours{D}{v}{u}|$, and the arc from~$v$ to~$u$ if $|\detours{D}{v}{u}| \geq |\detours{D}{u}{v}|$.
Notice that both directions are added in case of equality.
This process generates the \emph{detour completion}~$D^\star$ of~$D$ (\zcref{def:detours-and-completions}), on which we can apply the methods of \textcite{FominP19} since~$D^\star$ is semicomplete.

Second, we run on~$D^\star$ their algorithm~\cite[Theorem 4.12 and Lemmas 3.9 and 3.12]{FominP19} through~\zcref{prop:semicomplete-jungle-algorithm}.
It outputs a path decomposition of~$D^\star$ of small width, which is also a path decomposition of~$D$, or a \emph{short jungle} of~$D^\star$, consisting of a large set~$Z$ of vertices such that any two of them are joined by many internally disjoint paths of length at most four (\zcref{def:jungles}).
Notice that these paths may use arcs from $E(D^\star) \setminus E(D)$, which is why the direction of the added arcs matters.
For example, see the acyclic and \hsemi{1} digraph of~\zcref{fig:overview-detour-direction}.

\begin{figure}[ht]
  \centering
  \begin{tikzpicture}
  \foreach \i in {1,...,3} {
      \node[blackvertex, label={[yshift=-.1cm]-90:{$a_\i$}}] (a\i) at (1.5*\i-1.5, 0) {};
      \node[blackvertex, label={[yshift=.1cm]90:{$c_\i$}}] (c\i) at (1.5*\i-1.5, 1.7) {};
      \node[blackvertex, label={[yshift=.1cm]90:{$b_\i$}}] (b\i) at (1.2*\i+4.4, 0.85) {};
    }
  \foreach \i/\j in {1/2, 1/3, 2/1, 2/3, 3/1, 3/2} {
      \draw[arrow] (a\i) -- (c\j);
    }
  \foreach \i in {1,...,3} {
      \draw[arrow, very thick, dashed] (a\i) -- (c\i);
    }
  \node[rectangle, draw, inner sep=3pt, fit=(a1)(a3), label=180:{$A$}] (A) {};
  \node[rectangle, draw, inner sep=3pt, fit=(b1)(b3), label=0:{$B$}] (B) {};
  \node[rectangle, draw, inner sep=3pt, fit=(c1)(c3), label=180:{$C$}] (C) {};
  \draw[-{Latex[length=2.5mm]}, line width=1.2pt, gray] (A.east) -- (B.south west);
  \draw[-{Latex[length=2.5mm]}, line width=1.2pt, gray] (B.north west) -- (C.east);
\end{tikzpicture}%
%
  \caption{A digraph~$D$. It has all arcs from~$A$ to~$B$ and from~$B$ to~$C$, drawn as thick gray arrows, the arcs from~$a_i$ to~$c_j$ for $i \neq j$, and an acyclic tournament inside each of~$A$,~$B$, and~$C$, which is omitted. Its non-neighbors are the pairs~$a_i$,~$c_i$. The detour completion of~$D$ adds the dashed arcs~$(a_i, c_i)$, whereas a completion that adds the arcs~$(c_i, a_i)$ contains the $3$-triple~$(A, B, C)$.}
  \label{fig:overview-detour-direction}
\end{figure}
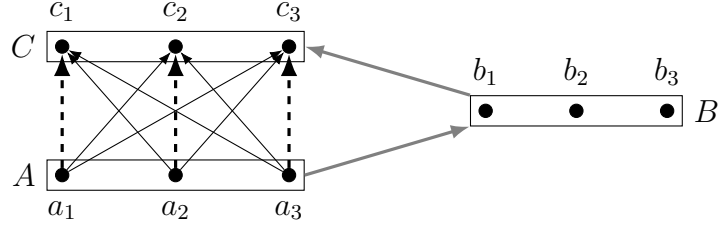

Third, we build a triple of~$D$ from the jungle~$Z$ by counting.
We first keep a subset~$Z'$ of~$Z$ that induces a semicomplete subgraph of~$D$, which loses a factor~$h + 1$ on the size (\zcref{lem:semicomplete-from-h-semicomplete}).
For a set~$X$ of vertices, we say that~$v$ is a \emph{common in-neighbor} of~$X$ if~$\{v\}$ is complete to~$X$, and a \emph{common out-neighbor} of~$X$ if~$X$ is complete to $\{v\}$.
Now, we find a set~$M \subseteq Z'$ of $8k$ vertices with many common in- and out-neighbors in~$Z'$ (\zcref{lem:sandwich}).
Any three vertices of a semicomplete digraph span a path of length two.
Hence, on average over the sets~$M$ of $8k$ vertices, many pairs~$(p, q)$ have every vertex of~$M$ as the middle of a detour from~$p$ to~$q$.
Namely,~$p$ is a common in-neighbor of~$M$ and~$q$ a common out-neighbor.
We choose the vertices of~$M$ one at a time, maintaining the average, so that the common in-neighbors and the common out-neighbors of~$M$ each keep a~$2^{-24k}$ fraction of~$Z'$.

Next, we look at a path of the jungle from a common out-neighbor of~$M$ to a common in-neighbor.
Every vertex~$x$ of~$M$ has an arc to the first vertex of the path and an arc from its last vertex, and, for all but few paths,~$x$ is adjacent to the vertices between.
Hence~$x$ is an in-neighbor of the tail~$c$ and an out-neighbor of the head~$a$ of some arc~$(c, a)$ of the path, and we say that~$x$ \emph{flips} at~$(c, a)$.
Since the path has at most four arcs, at least a quarter of~$M$, that is, $2k$ vertices, flip at the same arc~$(c, a)$, so~$a$ is complete to that quarter and the quarter is complete to~$c$ (\zcref{lem:flip}).
Notice that $(c, a)$ is then a candidate for a matching arc of a triple whose middle set lies in this quarter.
It remains to make the candidates agree.
We take $k \cdot 8^k$ paths whose candidate arcs have distinct ends.
A set of $2k$ vertices contains more than an~$8^{-k}$ fraction of the $k$-subsets of~$M$, so some~$k$ vertices of~$M$ lie in the quarters of~$k$ of these paths (\zcref{clm:common-subset}).
These~$k$ vertices form the middle set~$B$, the heads of the~$k$ candidate arcs form~$A$, and their tails form~$C$, so $(A, B, C)$ is a $k$-triple of~$D$ if the candidate arcs are arcs of~$D$ (\zcref{lem:jungle-extraction}).

Finally, we repair the candidate arcs that~$D^\star$ added.
If many candidate arcs~$(c, a)$ are not arcs of~$D$, then~$a$ and~$c$ are non-neighbors of~$D$, and the same count, now over the $2k$-subsets of~$M$, gives many such pairs~$(a, c)$ whose quarters share a set~$B$ of $2k$ vertices (\zcref{lem:jungle-extraction}).
Here the direction of the added arcs pays off.
Every vertex of~$B$ is the middle of a detour from~$a$ to~$c$, so $|\detours{D}{a}{c}| \geq 2k$, and since~$D^\star$ added the arc~$(c, a)$, also $|\detours{D}{c}{a}| \geq |\detours{D}{a}{c}| \geq 2k$.
A middle vertex~$u$ of a detour from~$c$ to~$a$ can stand in for~$a$ or for~$c$, since the arcs~$(c, u)$ and $(u, a)$ belong to~$D$.
Notice that, if~$u$ is adjacent to every vertex of~$B$, then~$u$ is complete to half of~$B$ or half of~$B$ is complete to~$u$, so~$u$ replaces~$a$ in the first case and~$c$ in the second.
The stand-ins of different pairs must be distinct, so a pair needs many detours from~$c$ to~$a$.
Unless the arcs of~$D$ from the vertices~$c$ to the vertices~$a$ already contain~$k$ disjoint arcs, which immediately give a triple, with K{\H{o}}nig's Theorem and by counting out-degrees we show that many pairs have many such detours (\zcref{thm:konig} and \zcref{obs:out-degree-h-semicomplete}).
Counting subsets of size~$k$ once more gives~$k$ pairs whose stand-ins agree on the same half of~$B$, which is the middle set of a $k$-triple of~$D$ (\zcref{lem:detour-pairs}).
Running the same extraction on another completion, the one with both arcs between non-neighbors, leads to the anti-triples of the next part.

\subparagraph*{Anti-triples.}
The dynamic programming of \textcite[Theorem 10.8]{FominP19} over a path decomposition of a semicomplete digraph needs semicompleteness only when a bag introduces a vertex~$v$, where it uses that every vertex forgotten so far is an out-neighbor of~$v$, which holds in semicomplete digraphs since a path decomposition has no arc from a forgotten vertex to~$v$.
So any path decomposition that \emph{covers the non-arcs}, that is, in which every two non-neighbors share a bag (\zcref{def:covering-non-arcs}), suffices to replicate this dynamic programming.
Since two vertices joined by arcs in both directions must share a bag, these decompositions are exactly the path decompositions of the completion~$\hat{D}$ of~$D$ that adds both arcs between every two non-neighbors (\zcref{sec:anti-triples}).
As~$\hat{D}$ is semicomplete, the approximation algorithm for semicomplete digraphs~\cite[Theorem 4.12]{FominP19} and the extraction of the previous part apply to it in place of~$D^\star$.
They give such a decomposition of small width, a triple of~$D$, or many pairs of non-neighbors~$a$,~$c$ and a set~$B$ with every~$a$ complete to~$B$ and~$B$ to every~$c$.
Since~$\hat{D}$ adds both arcs without comparing detours, the repair by detours of the previous part does not apply.
Instead, these pairs and~$B$ form an \emph{anti-triple}: a triple whose matching arcs from~$C$ to~$A$ are replaced by pairs of non-neighbors~$a_i$,~$c_i$ (\zcref{def:anti-triple,lem:covering-obstructions}).
As~$\hat{D}$ has the arcs~$(c_i, a_i)$, an anti-triple of~$D$ is a triple of~$\hat{D}$, so a $k$-anti-triple forces every decomposition covering the non-arcs to have width at least $k - 1$ (\zcref{cor:triple-dpw}).
Notice that an anti-triple may very well be acyclic if, for example, all arcs are oriented from~$A$ to~$C$, and thus it does \textsl{not} certify large directed pathwidth.
It is, however, detour-heavy from~$A$ to~$C$ by using the vertices of~$B$, and this fact is exploited in the algorithm for \ADP that follows.

\subparagraph*{Arc-disjoint paths.}
Assume we are given an instance~$\mathcal{I}$ of \ADP on an \hsemi{h} digraph~$D$ containing an anti-triple~$(A, B, C)$ with $|A|, |B|, |C| \geq 2k$, where~$k$ is the number of paths we want to route.
Let~$D'$ be~$D$ without the terminals of~$\mathcal{I}$, which the triples that we look for must avoid.
If $a \in A$ and $c \in C$ are non-neighbors in~$D$, then every vertex of~$B$ is the middle of a detour from~$a$ to~$c$.
Hence we may add the missing arc~$(a, c)$ without changing the answer to~$\mathcal{I}$, as observed by \textcite[(2.4)]{FradkinS15} (\zcref{prop:avoidable-arcs}): if a path of a solution uses~$(a, c)$, then the other paths use at most $k - 1$ arcs leaving~$a$ and at most $k - 1$ arcs entering~$c$, so some detour from~$a$ to~$c$ has both arcs unused, and the path can take it instead.
Adding arcs may increase the directed pathwidth, and then the algorithm may later find a large triple instead, from which we delete an irrelevant vertex.

\begin{figure}[ht]
  \centering
  \begin{tikzpicture}[
    font=\footnotesize,
    box/.style={draw, rounded corners=2pt, align=center, inner xsep=3pt, inner ysep=3pt},
    outbox/.style={box, fill=gray!12, minimum width=44mm, minimum height=10mm},
    act/.style={box, minimum width=44mm, minimum height=14.5mm},
    node distance=6mm and 4mm,
  ]
  \node[outbox] (o2) at (0, 0) {$\ell$-triple};
  \node[outbox, left=of o2] (o1) {decomposition of\\ width $\leq g(k, h)$};
  \node[outbox, right=of o2] (o3) {$\ell$-anti-triple\\ $(A, B, C)$};
  \node[box, above=7mm of o2, minimum width=140mm, minimum height=10mm] (round) {round: \zcref{lem:covering-obstructions} on $D'$ ($D$ without terminals),\\ approximation on $\widehat{D'}$ (\zcref{prop:semicomplete-jungle-algorithm}), then extraction (\zcref{lem:jungle-extraction})};
  \node[box, above=of round] (in) {instance~$\mathcal{I}$ on~$D$ with $k$ paths};
  \node[act, below=of o1] (a1) {terminals to all bags,\\ dynamic programming\\ (\zcref{lem:adp-dp}) solves~$\mathcal{I}$};
  \node[act, below=of o2] (a2) {delete irrelevant $x \in B$\\ (\zcref{lem:adp-irrelevant-vertex}):\\ $D := D \setminus \{x\}$};
  \node[act, below=of o3] (a3) {join non-neighbors $a \in A$,\\ $c \in C$ (\zcref{prop:avoidable-arcs}):\\ $D := D + (a, c)$};
  \draw[arrow] (in) -- (round);
  \foreach \i in {1, 2, 3} {
      \draw[arrow] (round.south -| o\i) -- (o\i);
      \draw[arrow] (o\i) -- (a\i);
    }
  \coordinate (r) at ($(round.east) + (3mm, 0)$);
  \coordinate (bot) at ($(a2.south) + (0, -4mm)$);
  \draw (a3.south) -- (a3.south |- bot);
  \draw[arrow] (a2.south) -- (bot) -- (bot -| r) -- (r) -- (round.east);
  \node[below=1mm, align=center] at ($(bot)!.5!(bot -| r)$) {next round: $D$ is smaller or denser,\\ $\Ocal((h + 1) \cdot n)$ rounds in total};
\end{tikzpicture}%
%
  \caption{One round of the algorithm of \zcref{thm:adp-fpt-simplified} for an instance~$\mathcal{I}$ of \ADP.}
  \label{fig:adp-rounds}
\end{figure}
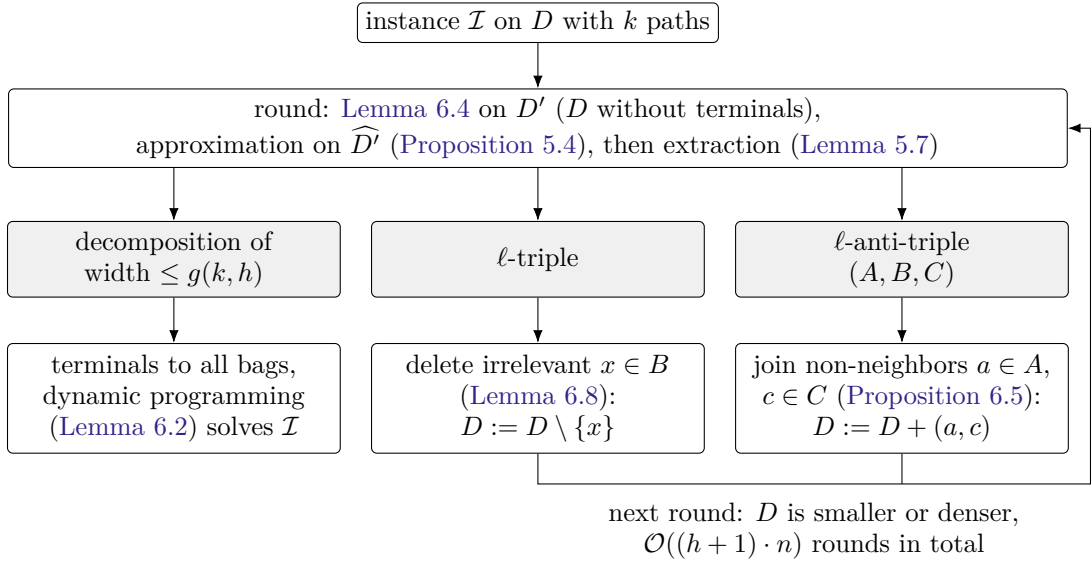

Our algorithm (\zcref{thm:adp-fpt-simplified}) repeats such steps in rounds, following the win-win strategy of \textcite[Theorem 9.9]{FominP19}, and keeps an instance equivalent to~$\mathcal{I}$ (\zcref{fig:adp-rounds}).
In each round, we test if~$D'$ has a path decomposition of width at most $g(k, h)$ that covers its non-arcs, that is, if~$\widehat{D'}$ has small directed pathwidth.
For this, we run on~$\widehat{D'}$ the approximation algorithm for semicomplete digraphs~\mbox{\cite[Theorem 4.12]{FominP19}} (or ours of \zcref{thm:tangles-or-path-decomposition} with $h = 0$), and turn a tangle that it may output into a short jungle~\cite[Lemmas 3.9 and 3.12]{FominP19} (\zcref{prop:semicomplete-jungle-algorithm}).
Our extraction (\zcref{lem:jungle-extraction}) then turns this short jungle into an $\ell$-triple or an $\ell$-anti-triple with $\ell = \Ocal(k^3 + kh)$.
The approximation algorithm runs with the size of the short jungle that our extraction needs, which is only single exponential in~$\ell$, so $g(k, h) = 2^{\Ocal(k^3 + kh)}$ (\zcref{lem:covering-obstructions}).
Now, each round outputs for~$D'$ (i)~a path decomposition of width at most $g(k, h)$ that covers its non-arcs, (ii)~an $\ell$-triple, or (iii)~an $\ell$-anti-triple.

In the first case, we add the terminals to every bag and solve~$\mathcal{I}$ by the dynamic programming (\zcref{lem:adp-dp}), which ends the algorithm.
In the second case, we delete a vertex of the triple that is irrelevant for~$\mathcal{I}$.
As for semicomplete digraphs~\mbox{\cite[Theorem 9.1]{FominP19}}, the paths through this vertex can be rerouted through the triple (\zcref{prop:solution-meets-triple}), but we choose the vertex with a counting lemma for \hsemi{h} digraphs (\zcref{lem:rich-vertices,lem:adp-irrelevant-vertex}).
In the third case, we add an arc between two non-neighbors of the anti-triple as above.
So each round either solves the instance, makes it smaller, or makes it denser.
Since~$D$ has~$n$ vertices and at most $hn/2$ pairs of non-neighbors (\zcref{obs:out-degree-h-semicomplete}), the instance is solved after $\Ocal((h + 1) \cdot n)$ rounds.
By the width bound above, the dynamic programming takes the doubly exponential time of \zcref{thm:adp-fpt-simplified}.

\section{Definitions and preliminaries}
\label{sec:prelim}

In this section, we fix the notation, define detours and completions, recall the results of \textcite{FominP19} that we use, and define path decompositions that cover the non-arcs and $k$-triples.

\subparagraph*{Digraphs.}
We refer to the book of \textcite{BondyM08} for background on graph theory.
Digraphs are finite, without loops or parallel arcs, but may have both arcs~$(u,v)$ and $(v,u)$.
A digraph~$D$ has vertex set~$V(D)$ and arc set~$E(D)$.
Digraphs are given by their adjacency matrices, so testing whether $(u,v) \in E(D)$ takes constant time.
For $X \subseteq V(D)$, $D[X]$ is the subgraph of~$D$ \emph{induced} by~$X$, $D \setminus X = D[V(D) \setminus X]$, and $[k] = \{1, \ldots, k\}$ for a positive integer~$k$.
An arc~$(u,v)$ has \emph{tail}~$u$ and \emph{head}~$v$, and~$u$ is an \emph{in-neighbor} of~$v$ and~$v$ an \emph{out-neighbor} of~$u$.
We write $N^+(u)$ for the set of out-neighbors of~$u$, $\deg^+(u) = |N^+(u)|$ for its \emph{out-degree}, and $N^-(v)$ for the set of in-neighbors of~$v$.
The \emph{in-degree} of a vertex is its number of in-neighbors.
Distinct vertices~$u,v$ are \emph{adjacent} if $(u,v) \in E(D)$ or $(v,u) \in E(D)$, and otherwise they are \emph{non-neighbors} and $(u,v)$ is a \emph{non-arc} of~$D$.
A set~$X \subseteq V(D)$ is \emph{complete to} $Y \subseteq V(D)$ if $(x,y) \in E(D)$ for all $x \in X$ and $y \in Y$, where a vertex~$v$ stands for $\{v\}$.
The set~$X$ is \emph{matched to}~$Y$ if $t = |X| = |Y|$ and there are orderings~$(x_1, \ldots, x_t)$ of~$X$ and $(y_1, \ldots, y_t)$ of~$Y$ with $(x_i, y_i) \in E(D)$ for every $i \in [t]$.

A digraph is \emph{acyclic} if it has an ordering~$(v_1, \ldots, v_n)$ of its vertices with $i < j$ for every arc~$(v_i, v_j)$.
It is \emph{semicomplete} if it has no non-arcs, and a \emph{tournament} if, in addition, it has exactly one arc between any two distinct vertices.
For a non-negative integer~$h$, a digraph is \emph{\hsemi{h}} if every vertex has at most~$h$ non-neighbors.
Induced subgraphs of \hsemi{h} digraphs are \hsemi{h}, and the following lemma finds large semicomplete ones.

\begin{lemma}\label{lem:semicomplete-from-h-semicomplete}
  Let~$k$ be a positive integer,~$h$ be a non-negative integer, and $\fref{lem:semicomplete-from-h-semicomplete}(k, h) = k(h+1)$.
  Let~$D$ be an \hsemi{h} digraph on $n \geq \fref{lem:semicomplete-from-h-semicomplete}(k, h)$ vertices.
  Then there is a set~$X \subseteq V(D)$ of size~$k$ such that $D[X]$ is semicomplete, which can be found in $\Ocal(k \cdot n)$ time.
\end{lemma}

\begin{proof}
  Start with $X = \emptyset$ and $Y = V(D)$.
  While $|X| < k$, move a vertex~$v$ from~$Y$ to~$X$ and remove all non-neighbors of~$v$ from~$Y$.
  Every vertex of~$Y$ stays adjacent to every vertex of~$X$, so $D[X]$ is semicomplete.
  Each step removes at most $h + 1$ vertices from~$Y$, so $n \geq k(h+1)$ keeps~$Y$ non-empty until $|X| = k$.
  Each step takes $\Ocal(n)$ time, and the result follows.
\end{proof}

We use the following count of non-arcs and arcs in \zcref{sec:large-dpw-to-triples,sec:adp-algorithm}.

\begin{observation}\label{obs:out-degree-h-semicomplete}
  Every \hsemi{h} digraph on $N \geq 1$ vertices has at most $hN/2$ pairs of non-neighbors and a vertex of out-degree at least $(N - 1 - h)/2$.
\end{observation}

\begin{proof}
  Each vertex belongs to at most~$h$ pairs of non-neighbors, which gives the first bound.
  Thus the digraph has at least $\binom{N}{2} - hN/2 = N(N - 1 - h)/2$ arcs, and the out-degrees of the~$N$ vertices sum to the number of arcs, so one of them is at least $(N - 1 - h)/2$, and the result follows.
\end{proof}

A \emph{walk} from~$p_0$ to~$p_r$ is a sequence~$(p_0, \ldots, p_r)$ of vertices with $(p_{i-1}, p_i) \in E(D)$ for every $i \in [r]$, and a \emph{path} if its vertices are distinct.
Its \emph{length} is~$r$, its \emph{arcs} are $(p_{i-1}, p_i)$ for $i \in [r]$, and $p_1, \ldots, p_{r-1}$ are its \emph{internal vertices}.
Paths are \emph{arc-disjoint} if no arc belongs to two of them, and \emph{internally disjoint} if they are pairwise distinct and no vertex is an internal vertex of two of them.

\begin{definition}[Detours and completions]\label{def:detours-and-completions}
  Let~$D$ be a digraph.
  For distinct vertices~$u$ and~$v$ of~$D$, a \emph{detour} from~$u$ to~$v$ is a path of length two from~$u$ to~$v$, and $\detours{D}{u}{v}$ is the set~$N^+(u) \cap N^-(v)$.
  A \emph{completion} of~$D$ is a semicomplete digraph~$S$ with $V(S) = V(D)$ and $E(D) \subseteq E(S)$ such that every arc of~$S$ that is not an arc of~$D$ joins two non-neighbors of~$D$.
  The \emph{detour completion}~$D^\star$ of~$D$ is the digraph obtained from~$D$ by adding the arc~$(u, v)$ for every non-arc~$(u, v)$ of~$D$ with $|\detours{D}{u}{v}| \geq |\detours{D}{v}{u}|$.
\end{definition}

Since~$D^\star$ has the arc~$(u, v)$ or $(v, u)$ for every non-arc~$(u, v)$ of~$D$, it is a completion of~$D$, and it has both arcs between non-neighbors~$u$ and~$v$ of~$D$ exactly when $|\detours{D}{u}{v}| = |\detours{D}{v}{u}|$.
\zcref{fig:detour-completion} shows an added arc.
We use the detour completion in \zcref{sec:large-dpw-to-triples} and, in \zcref{sec:anti-triples}, the completion that adds both arcs between every two non-neighbors.

\begin{figure}[ht]
  \centering
  \begin{tikzpicture}
  \node[blackvertex, label={[xshift=-.1cm]180:{$u$}}] (u) at (0, 0) {};
  \node[blackvertex, label={[xshift=.1cm]0:{$v$}}] (v) at (5, 0) {};
  \foreach \i/\x in {1/1.5, 2/2.5, 3/3.5} {
      \node[blackvertex] (w\i) at (\x, 1.4) {};
      \draw[arrow] (u) -- (w\i);
      \draw[arrow] (w\i) -- (v);
    }
  \node[blackvertex] (w4) at (2.5, -1.4) {};
  \draw[arrow] (v) -- (w4);
  \draw[arrow] (w4) -- (u);
  \draw[arrow, dashed] (u) -- (v);
  \node[rectangle, draw, inner sep=5pt, fit=(w1)(w3), label=90:{$\detours{D}{u}{v}$}] {};
  \node[rectangle, draw, inner sep=5pt, fit=(w4), label=-90:{$\detours{D}{v}{u}$}] {};
\end{tikzpicture}%
%
  \caption{Two non-neighbors~$u$ and~$v$ of a digraph~$D$ with $|\detours{D}{u}{v}| = 3$ and \mbox{$|\detours{D}{v}{u}| = 1$}. The detour completion of~$D$ adds the arc~$(u, v)$, drawn dashed.}
  \label{fig:detour-completion}
\end{figure}
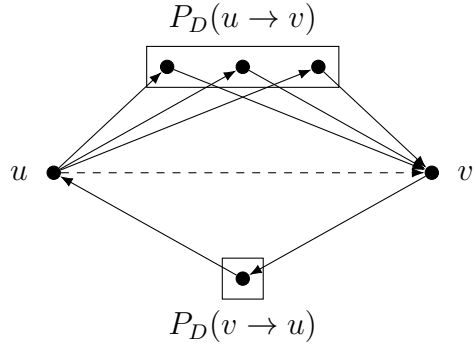

A \emph{separation} of~$D$ is a pair~$(A, B)$ of subsets of $V(D)$ with $A \cup B = V(D)$ and no arc from $A \setminus B$ to $B \setminus A$.
Its \emph{separator} is $A \cap B$ and its \emph{order} is $|A \cap B|$.

An undirected graph~$G$ is \emph{bipartite} with \emph{bipartition}~$(X, Y)$ if~$X$ and~$Y$ partition~$V(G)$ and every edge joins~$X$ to~$Y$.
A \emph{matching} of~$G$ is a set of pairwise disjoint edges, $\nu(G)$ is the maximum size of a matching of~$G$, and a \emph{vertex cover} of~$G$ is a set of vertices meeting every edge.

\begin{theorem}[Kőnig's Theorem~\cite{Konig31}]\label{thm:konig}
  Every bipartite graph~$G$ has a vertex cover of size~$\nu(G)$.
\end{theorem}

Given a bipartite graph~$G$ on~$n$ vertices and a positive integer~$k$, enlarging a matching~$k$ times along augmenting paths~\cite[Chapter 16]{BondyM08} finds in $\Ocal(k \cdot n^2)$ time a matching of size~$k$ or a vertex cover of size at most $k - 1$.

\subparagraph*{Path decompositions.} We now formally define directed pathwidth and path decompositions.

\begin{definition}[Directed pathwidth]\label{def:pathwidth}
  A \emph{path decomposition} of a digraph~$D$ is a sequence~$\mathcal{X} = (X_1, \ldots, X_r)$ of subsets of $V(D)$, called \emph{bags}, such that
  \begin{outline}[romanenumerate]
    \1 $\bigcup_{i \in [r]} X_i = V(D)$,
    \1 for every arc~$(u, v) \in E(D)$, we have $u, v \in X_i$ for some $i \in [r]$, or $u \in X_i$ and $v \in X_j$ for some $i > j$, and\label{def:pathwidth-item-2}
    \1 for every vertex~$v \in V(D)$, the set~$\{i \in [r] \mid v \in X_i\}$ is an interval of integers.
  \end{outline}
  The \emph{width} of~$\mathcal{X}$ is $\max_{i \in [r]} |X_i| - 1$, and the \emph{directed pathwidth}~$\dpw(D)$ of~$D$ is the minimum width of a path decomposition of~$D$.
  The path decomposition~$\mathcal{X}$ is \emph{nice} if $X_1 = X_r = \emptyset$ and, for every $i \in [r-1]$, either $X_{i+1} = X_i \cup \{v\}$ for some $v \notin X_i$, or $X_{i+1} = X_i \setminus \{v\}$ for some $v \in X_i$.
  In the first case~$X_{i+1}$ \emph{introduces}~$v$, and in the second it \emph{forgets}~$v$.
\end{definition}

\begin{observation}\label{obs:dpw-subgraph}
  Let~$D'$ be a subgraph of a digraph~$D$.
  Then $\dpw(D') \leq \dpw(D)$, and if $V(D') = V(D)$, then every path decomposition of~$D$ is one of~$D'$.
\end{observation}

\begin{proof}
  Restricting every bag of a path decomposition of~$D$ to $V(D')$ gives a path decomposition of~$D'$ of at most the same width, since every arc of~$D'$ is an arc of~$D$.
  If $V(D') = V(D)$, the restriction changes no bag, and the result follows.
\end{proof}

Thus an obstruction in a subgraph of~$D$ bounds $\dpw(D)$ from below, and every path decomposition of a completion of~$D$ is one of~$D$.

A path decomposition of~$D$ with~$r$ bags and width~$k$ becomes nice, with the same width, by removing repeated bags and then inserting bags~\cite[Section 2.4.2]{FominP19}, in $\Ocal((k + 1) \cdot (n + r))$ time, where $n = |V(D)|$.
An acyclic digraph~$D$ has $\dpw(D) = 0$, since $(\{v_n\}, \ldots, \{v_1\})$ is a path decomposition of~$D$ for the ordering in the definition of acyclic digraphs.
Path decompositions give separations of~$D$ as follows, and all our lower bounds on directed pathwidth use them.

\begin{proposition}[{\textcite[Lemma 2.12 and Section 2.4.2]{FominP19}}]
  \label{prop:separations-from-nice-path-decomposition}
  Let $\mathcal{X} = (X_1, \ldots, X_r)$ be a path decomposition of width at most~$k$ of a digraph~$D$ on~$n$ vertices.
  For $i \in \{0, \ldots, r\}$, let $A_i = X_1 \cup \cdots \cup X_i$ and $B_i = X_{i+1} \cup \cdots \cup X_r$.
  Then $(A_i, B_i)$ is a separation of~$D$ for every $i \in \{0, \ldots, r\}$.
  If~$\mathcal{X}$ is nice, then for every $j \in \{0, \ldots, n\}$ there is $i \in \{0, \ldots, r\}$ with $|A_i| = j$ and $|A_i \cap B_i| \leq k$.
\end{proposition}

In a nice path decomposition of a semicomplete digraph, every vertex not yet introduced has an arc to every vertex already forgotten~\cite[Section 2.4.2]{FominP19}, and path decompositions put no condition on the non-arcs.
We introduce the following notion to recover this property in \zcref{sec:adp-algorithm,sec:cliquewidth-and-pathwidth}.

\begin{definition}[Covering the non-arcs]\label{def:covering-non-arcs}
  A path decomposition~$\mathcal{X}$ of a digraph~$D$ \emph{covers the non-arcs} of~$D$ if, for every non-arc~$(u, v)$ of~$D$, some bag of~$\mathcal{X}$ contains both~$u$ and~$v$.
\end{definition}

Making a path decomposition nice as above only removes bags and inserts subsets of its bags, so it preserves the covering of the non-arcs.
The proof of the following observation is the argument for semicomplete digraphs~\cite[Section 2.4.2]{FominP19}, with the covering of the non-arcs in place of semicompleteness.

\begin{observation}\label{obs:covering-non-arcs}
  Let $\mathcal{X} = (X_1, \ldots, X_r)$ be a path decomposition of a digraph~$D$ that covers the non-arcs of~$D$, and~$A_i$ and~$B_i$ be as in \zcref{prop:separations-from-nice-path-decomposition}.
  Then $B_i \setminus A_i$ is complete to $A_i \setminus B_i$ for every $i \in \{0, \ldots, r\}$.
  If~$\mathcal{X}$ is nice and~$X_{i+1}$ introduces a vertex~$v$, then $v \in B_i \setminus A_i$, the set~$A_i \setminus B_i = A_i \setminus X_i$ consists of the vertices forgotten by $X_1, \ldots, X_i$, and~$v$ is complete to it.
\end{observation}

\begin{proof}
  Let $v \in B_i \setminus A_i$ and $w \in A_i \setminus B_i$.
  Then~$w$ belongs only to bags~$X_j$ with $j \leq i$, and~$v$ only to bags with $j > i$, so no bag contains both.
  Since~$\mathcal{X}$ covers the non-arcs of~$D$, the vertices~$v$ and~$w$ are adjacent.
  As $(A_i, B_i)$ is a separation of~$D$ by \zcref{prop:separations-from-nice-path-decomposition}, there is no arc from~$w$ to~$v$, and we conclude that $(v, w) \in E(D)$.
  If~$\mathcal{X}$ is nice and~$X_{i+1}$ introduces~$v$, then~$v$ is in no bag before~$X_{i+1}$, so $v \in B_i \setminus A_i$, and $A_i \cap B_i = X_i$ since $X_i \subseteq X_{i+1}$ and the bags containing a vertex form an interval.
  Thus $A_i \setminus B_i = A_i \setminus X_i$ consists of the vertices forgotten by $X_1, \ldots, X_i$.
  By the first statement,~$v$ is complete to it, and the result follows.
\end{proof}

\begin{definition}[$k$-triple {\cite[Section 1]{FradkinS13}}]\label{def:k-triple}
  Let~$D$ be a digraph and~$k$ be a positive integer.
  A \emph{$k$-triple} of~$D$ is a triple~$(A, B, C)$ of pairwise disjoint subsets of $V(D)$ of size~$k$ such that~$A$ is complete to~$B$,~$B$ is complete to~$C$, and~$C$ is matched to~$A$.
\end{definition}

See \zcref{fig:k-triple-example} for a $4$-triple.

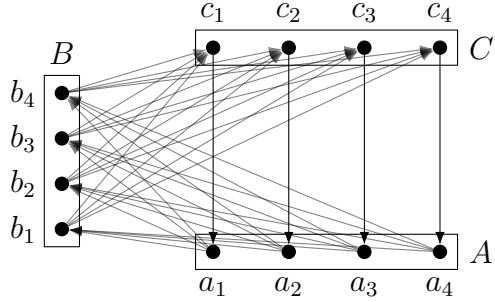
\begin{figure}[ht]
  \centering
  \begin{tikzpicture}[yscale=.6]
  \foreach \i in {1,...,4} {
      \node[blackvertex, label={[yshift=-.1cm]-90:{$a_\i$}}] (a\i) at (\i-1,-.5) {};
      \node[blackvertex, label={[xshift=-.1cm]180:{$b_\i$}}] (b\i) at (-2,\i-1) {};
      \node[blackvertex, label={[yshift=.1cm]90:{$c_\i$}}] (c\i) at ($(a\i) + (0,4.5)$) {};
      \draw[arrow] (c\i) -- (a\i);
    }
  \foreach \i in {1,...,4}
  \foreach \j in {1,...,4} {
      \draw[arrow,opacity=.5] (a\i) -- (b\j);
      \draw[arrow,opacity=.5] (b\i) -- (c\j);
    }
  \node[rectangle,draw,fit=(a1)(a4), label=0:{$A$}] {};
  \node[rectangle,draw,fit=(b1)(b4), label=90:{$B$}] {};
  \node[rectangle,draw,fit=(c1)(c4), label=0:{$C$}] {};
\end{tikzpicture}%
%
  \caption{A $4$-triple~$(A, B, C)$. The definition puts no condition on the arcs inside~$A$,~$B$, and~$C$ or on the other arcs between~$C$ and~$A$.}
  \label{fig:k-triple-example}
\end{figure}

\subparagraph*{Parameterized complexity.}
A \emph{parameterized problem} is a set of pairs~$(x, k)$, called \emph{instances}, of a string~$x$ and a non-negative integer~$k$, the \emph{parameter}.
It is \emph{fixed-parameter tractable} (\FPT for short) if some algorithm decides it in $f(k) \cdot {|x|}^{\Ocal(1)}$ time, and it is in the class~\XP if some algorithm decides it in $f(k) \cdot {|x|}^{g(k)}$ time, for computable functions~$f$ and~$g$.
We refer to the book of \textcite{CyganFKLMPPS15} for background on parameterized complexity.
For a non-negative integer~$i$, we write $\tower_i(x)$ for the tower of~$i$ twos with~$x$ on top, so $\tower_0(x) = x$ and $\tower_{i+1}(x) = 2^{\tower_i(x)}$.

\section{A polynomial-time approximation algorithm for directed pathwidth in \titlemath{h}-semicomplete digraphs}
\label{sec:polytime-app-algorithm}

In this section, we extend the polynomial-time approximation algorithm of \textcite[Theorem 4.12]{FominP19} for directed pathwidth from semicomplete to \hsemi{h} digraphs.
Their algorithm outputs, for a parameter~$\ell \geq 4k$, a path decomposition of width at most $\ell + 2k$, a degree tangle, or a matching tangle.
We show that tangles remain obstructions to directed pathwidth at most~$k$ in \hsemi{h} digraphs, with a slack of~$h$ in their out-degrees (\zcref{thm:tangles-or-path-decomposition}).
We first define tangles~\mbox{\cite[Definitions~3.7 and 3.10]{FominP19}}.

\begin{definition}[$(s, d)$-degree tangles]
\label{def:degree-tangles}
  Let~$D$ be a digraph and $s, d$ be positive integers.
  An \emph{$(s, d)$-degree tangle} of~$D$ is a set~$U \subseteq V(D)$ such that $|U| \geq s$ and $|\deg^+(v) - \deg^+(w)| \leq d$ for all $v, w \in U$.
\end{definition}

\begin{definition}[$(s, d)$-matching tangles]
\label{def:matching-tangles}
  Let~$D$ be a digraph and $s, d$ be positive integers.
  An \emph{$(s, d)$-matching tangle} of~$D$ is a pair~$(U, W)$ of disjoint subsets of $V(D)$ of size~$s$ such that~$U$ is matched to~$W$ and $\deg^+(w) > \deg^+(u) + d$ for all $u \in U$ and $w \in W$.
\end{definition}

\zcref{prop:fp-approximation} is adapted from the approximation algorithm for semicomplete digraphs~\cite[Theorem 4.12]{FominP19}.
Although the original statement is restricted to semicomplete digraphs, we remark that all the tools that its algorithm applies hold for general digraphs.
The proof orders the vertices by out-degree, finds the tangles or small vertex covers of bipartite graphs formed by arcs, and turns a monotone sequence of separations into a path decomposition, using their results stated for general digraphs and bipartite graphs.
A monotone sequence of separations yields a path decomposition with bounded width~\cite[Lemma 2.12]{FominP19}, and the vertices of a bipartite graph~$G$ covering an edge of every maximum matching form a vertex cover with at most $\nu(G)$ vertices on each side of~$G$ that changes monotonically when vertices are added to one side of~$G$ and removed from the other side~\cite[Lemmas 4.2 and 4.4]{FominP19}.
Semicompleteness is used in the original proof~\cite[Theorem 4.12]{FominP19} only to guarantee its last conclusion, which depends on degree and matching tangles being certificates of large directed pathwidth of semicomplete digraphs.
In~\zcref{prop:fp-approximation} we drop this last conclusion from the statement and write it for general digraphs.

\begin{proposition}[Adapted from {\textcite[Theorem 4.12]{FominP19}}]
\label{prop:fp-approximation}
  Let~$D$ be a digraph on~$n$ vertices, and~$k$ and $\ell \geq 4k$ be positive integers.
  There is an algorithm with run time~$\Ocal(k \cdot n^2)$ that outputs an $(\ell + 2, k)$-degree tangle of~$D$, a $(k + 1, k)$-matching tangle of~$D$, or a path decomposition of~$D$ of width at most $\ell + 2k$.
\end{proposition}

In~\zcref{lem:degree-tangles-large-dpw} and \zcref{lem:matching-tangles-large-dpw}, we show that both tangles remain certificates for large directed pathwidth of \hsemi{h} digraphs.
We apply~\zcref{prop:fp-approximation} with $k + h$ in place of~$k$ to account for $h$-semicompleteness, and we use the two aforementioned lemmas to recover the original conclusion, which proves the following:

\begin{theorem}[store=tangles-or-path-decomposition]
\label{thm:tangles-or-path-decomposition}
  Let~$D$ be an $n$-vertex \hsemi{h} digraph, and let $k,\ell \geq 4(k + h)$ be positive integers.
  There is an algorithm that, in time~$\Ocal((k + h) \cdot n^2)$, outputs one of the following:
  \begin{outline}[enumerate]
    \1 an $(\ell + 2, k + h)$-degree tangle of~$D$,
    \1 a $(k + 1, k + h)$-matching tangle of~$D$, or
    \1 a path decomposition of~$D$ of width at most $\ell + 2(k + h)$.
  \end{outline}
  In the first two cases, $\dpw(D) > k$.
\end{theorem}

With $h = 0$ we recover exactly the result for semicomplete digraphs~\cite[Theorem~4.12]{FominP19}.
The slack of~$h$ in the tangles comes from the following fact.
Given a separation~$(A, B)$ of~$D$, a vertex~$b \in B \setminus A$ has as out-neighbors all \emph{but at most~$h$} vertices of $A \setminus B$, instead of all of them in a semicomplete digraph.
This forces us to ask for a $(k + 1, k + h)$-matching tangle, \emph{stronger} than the $(k + 1, k)$-matching tangle for semicomplete digraphs~\cite{FominP19}, see \zcref{lem:matching-tangles-large-dpw}.
In turn, the algorithm uses the same bound~$k + h$ for the degree tangle, which is thus \emph{weaker} and needs $2h$ more vertices.

\raggedbottom
\begin{corollary}\label{cor:dpw-approximation}
  Let~$D$ be an \hsemi{h} digraph on~$n$ vertices, and let $p = \max\{\dpw(D), 1\}$.
  There is an algorithm with run time~$\Ocal(p \cdot (p + h) \cdot n^2)$ that outputs a path decomposition of~$D$ of width at most $6(p + h)$.
\end{corollary}

\begin{proof}
  Run the algorithm of \zcref{thm:tangles-or-path-decomposition} with $\ell = 4(k + h)$ for $k = 1, 2, \ldots$ until it outputs a path decomposition.
  The tangles are impossible once $k \geq \dpw(D)$, so the algorithm stops for some $k \leq p$ with width at most $6(k + h) \leq 6(p + h)$.
  The $k$-th run takes $\Ocal((k + h) \cdot n^2)$ time, and the result follows.
\end{proof}

For $\dpw(D) \geq 1$, this width is at most $6(h + 1) \cdot \dpw(D)$, a constant-factor approximation for fixed~$h$.

The proofs that the tangles are obstructions use \zcref{prop:separations-from-nice-path-decomposition} and the following result, stated for semicomplete digraphs~\cite[Lemma 2.10]{FominP19}; its proof, which takes the vertex introduced last in a nice path decomposition, works without change in every digraph.

\begin{proposition}[Adapted from {\textcite[Lemma 2.10]{FominP19}}]
\label{prop:small-in-degree}
  Every digraph with at least one vertex and directed pathwidth at most~$k$ has a vertex of in-degree at most~$k$.
\end{proposition}

We are ready to show that degree tangles are obstructions to bounded directed pathwidth.
The proof follows the one for semicomplete digraphs~\cite[Lemma 3.8]{FominP19}.
Our tangle has $2h$ more vertices than the original one:~$h$ compensate the gap~$k + h$ between the out-degrees in the tangle, which is~$k$ in the original lemma, and~$h$ account for the non-neighbors of the vertex whose out-degree we bound.

\begin{lemma}
\label{lem:degree-tangles-large-dpw}
  Let~$D$ be an \hsemi{h} digraph and~$k$ be a positive integer.
  If~$D$ has a $(4k + 2h + 2, k + h)$-degree tangle, then $\dpw(D) > k$.
\end{lemma}

\begin{proof}
  Let~$U$ be such a tangle, $\alpha = \min_{v \in U} \deg^+(v)$, and $\beta = \max_{v \in U} \deg^+(v)$, so $\beta - \alpha \leq k + h$.
  By contradiction, assume that~$D$ has a nice path decomposition of width at most~$k$.
  By \zcref{prop:separations-from-nice-path-decomposition},~$D$ has a separation~$(A, B)$ with $|A \cap B| \leq k$ and $|A| = \alpha$.
  A vertex~$v \in A \setminus B$ has all its out-neighbors in $A \setminus \{v\}$, so $\deg^+(v) \leq \alpha - 1$ and $U \cap (A \setminus B) = \emptyset$.
  Since $|A \cap B| \leq k$, $D' = D[U \cap (B \setminus A)]$ has at least $3k + 2h + 2$ vertices, and $\dpw(D') \leq k$ since~$D'$ is an induced subgraph of~$D$ (\zcref{obs:dpw-subgraph}).
  By \zcref{prop:small-in-degree}, there is $u \in V(D')$ with at most~$k$ in-neighbors in~$D'$.
  Since $(A, B)$ is a separation, every vertex of $A \setminus B$ adjacent to~$u$ is an out-neighbor of~$u$, and so is every vertex of $V(D') \setminus \{u\}$ adjacent to~$u$ that is not an in-neighbor.
  Since~$u$ has at most~$h$ non-neighbors, we conclude that
  \begin{align*}
    \deg^+(u) & \geq |A \setminus B| + (|V(D')| - 1 - k) - h                       \\
              & \geq (\alpha - k) + (2k + 2h + 1) - h = \alpha + k + h + 1 > \beta\enspace.
  \end{align*}
  This contradicts $u \in U$, and the result follows.
\end{proof}

\flushbottom
Next, we consider matching tangles.
The proof follows the one for semicomplete digraphs~\cite[Lemma 3.11]{FominP19}, and the only change is that the gap~$k + h$ between the out-degrees absorbs the at most~$h$ non-neighbors of a vertex of~$U$.

\begin{lemma}
\label{lem:matching-tangles-large-dpw}
  Let~$D$ be an \hsemi{h} digraph and~$k$ be a positive integer.
  If~$D$ has a $(k + 1, k + h)$-matching tangle, then $\dpw(D) > k$.
\end{lemma}

\begin{proof}
  Let $(U, W)$ be such a tangle and $\alpha = \min_{w \in W} \deg^+(w)$.
  By contradiction, assume that~$D$ has a nice path decomposition of width at most~$k$.
  By \zcref{prop:separations-from-nice-path-decomposition},~$D$ has a separation~$(A, B)$ with $|A \cap B| \leq k$ and $|A| = \alpha$.
  First, $W \subseteq B$: as above, a vertex~$w \in A \setminus B$ has $\deg^+(w) \leq \alpha - 1$.
  Next, we claim that $U \subseteq A$.
  Otherwise, let $u \in U \cap (B \setminus A)$.
  Since $(A, B)$ is a separation and~$D$ is \hsemi{h}, all but at most~$h$ vertices of $A \setminus B$ are out-neighbors of~$u$, so $\deg^+(u) \geq |A \setminus B| - h \geq \alpha - k - h$.
  Then every $w \in W$ has $\deg^+(w) > \deg^+(u) + k + h \geq \alpha$, contradicting the choice of~$\alpha$.
  Finally, at most~$k$ of the $k + 1$ arcs matching~$U$ to~$W$ have an endpoint in $A \cap B$.
  Any other one goes from $A \setminus B$ to $B \setminus A$, contradicting that $(A, B)$ is a separation, and the result follows.
\end{proof}

\begin{proof}[Proof of \zcref{thm:tangles-or-path-decomposition}]
  We run the algorithm of \zcref{prop:fp-approximation} with $k + h$ in place of~$k$, which is possible as $\ell \geq 4(k + h)$.
  In $\Ocal((k + h) \cdot n^2)$ time, it outputs a path decomposition of width at most $\ell + 2(k + h)$, an $(\ell + 2, k + h)$-degree tangle, which contains a $(4k + 2h + 2, k + h)$-degree tangle as $\ell + 2 \geq 4k + 2h + 2$, or a $(k + h + 1, k + h)$-matching tangle~$(U, W)$.
  Keeping $k + 1$ of the matching arcs and their endpoints gives a $(k + 1, k + h)$-matching tangle.
  In the last two cases, $\dpw(D) > k$ by \zcref{lem:degree-tangles-large-dpw,lem:matching-tangles-large-dpw}, and the result follows.
\end{proof}

\section{Finding \titlemath{k}-triples in \titlemath{h}-semicomplete digraphs of large directed pathwidth}
\label{sec:large-dpw-to-triples}

In this section, we show that every \hsemi{h} digraph on~$n$ vertices of directed pathwidth greater than $(h+1) \cdot k \cdot 2^{27k+9}$ contains a $k$-triple, and we find it in $\min\{{(h+1)}^3 \cdot 2^{\Ocal(k)} \cdot n^2, \Ocal(n^5)\}$ time, which is polynomial in~$n$ (\zcref{thm:large-dpw-to-triples}).
We run an algorithm of \textcite[proof of Theorem 9.9]{FominP19} on the detour completion of the input digraph~$D$, which outputs a path decomposition of bounded width or a short jungle.
We turn the short jungle into a $k$-triple of~$D$ in \zcref{sec:jungle-to-triple}, and we close the section with lower bounds showing that the linear dependence on~$h$ is best possible and the dependence on~$k$ is best possible up to the constant in the exponent (\zcref{lem:large-dpw-to-triples-tight}).

\subsection{Jungles and the detour completion}\label{sec:jungles-detour-completion}

\begin{definition}[Jungles {\cite[Section 1]{FradkinS13}} and short jungles {\cite[Definition 3.2]{FominP19}}]\label{def:jungles}
  Let~$D$ be a digraph and~$k$ be a positive integer.
  A set~$U \subseteq V(D)$ is a \emph{jungle of size~$k$} if $|U| = k$ and, for all distinct $v, w \in U$, either $(v, w) \in E(D)$ or there are~$k$ internally disjoint paths from~$v$ to~$w$ in~$D$.
  For a positive integer~$d$, the set~$U$ is a \emph{$(k, d)$-short jungle} if $|U| \geq k$ and, for all distinct $v, w \in U$, there are~$k$ internally disjoint paths from~$v$ to~$w$ in~$D$ of length at most~$d$.
\end{definition}

A short jungle is always given or output with the paths required by its definition, as given by \textcite[Section 3.1]{FominP19}.
They also observe that, in a semicomplete digraph, a path decomposition of small width would have a bag containing a whole jungle; their argument works in every digraph, and we include a proof for completeness.

\begin{lemma}\label{lem:jungle-dpw}
  Let~$t$ be a positive integer.
  If a digraph~$D$ contains a jungle of size~$t$, then $\dpw(D) \geq t - 1$.
\end{lemma}

\begin{proof}
  Let~$U$ be a jungle of size~$t$ in~$D$ and $(X_1, \ldots, X_r)$ be a path decomposition of~$D$.
  Let~$i$ be the smallest index such that some $v \in U$ belongs to~$X_i$ and to no later bag.
  Let $A = X_1 \cup \cdots \cup X_i$ and $B = X_{i+1} \cup \cdots \cup X_r$.
  By \zcref{prop:separations-from-nice-path-decomposition}, $(A, B)$ is a separation of~$D$, and $A \cap B \subseteq X_i$ since the bags containing a given vertex form an interval.
  If $U \subseteq X_i$, then $|X_i| \geq t$.
  Otherwise, let $u \in U \setminus X_i$.
  Since the bags containing~$u$ form an interval,~$u$ belongs only to bags before~$X_i$ or only to bags after~$X_i$, and the first option contradicts the choice of~$i$.
  Thus $u \in B \setminus A$, and $v \in A \setminus B$ as~$v$ belongs to no bag after~$X_i$, so $(v, u) \notin E(D)$ and there are~$t$ internally disjoint paths from~$v$ to~$u$ in~$D$.
  The first vertex of each of these paths outside $A \setminus B$ is in $A \cap B \subseteq X_i$, since there is no arc from $A \setminus B$ to $B \setminus A$, and it is an internal vertex since it is neither~$v$ nor~$u$.
  These~$t$ vertices are distinct and differ from~$v$, so $|X_i| \geq t + 1$, and the result follows.
\end{proof}

The middle set~$B$ of a $k$-triple is a jungle of size~$k$, since for distinct $b, b' \in B$ the~$k$ matching arcs~$(c_i, a_i)$ give the internally disjoint paths~$(b, c_i, a_i, b')$, so by \zcref{lem:jungle-dpw} $k$-triples are obstructions as well.

\begin{corollary}\label{cor:triple-dpw}
  Let~$k$ be a positive integer.
  If a digraph~$D$ contains a $k$-triple, then $\dpw(D) \geq k - 1$.
\end{corollary}

The following algorithm is the first step of the algorithm for \textsc{Rooted Immersion} in semicomplete digraphs~\cite[proof of Theorem~9.9]{FominP19}; see also the proof in Pilipczuk's PhD thesis~\cite[Theorem 106]{Pilipczuk13}.
It runs the approximation algorithm for semicomplete digraphs~\cite[Theorem 4.12]{FominP19}, or ours of \zcref{thm:tangles-or-path-decomposition} with $h = 0$, with $5K$ and $130K$ in place of the parameters~$k$ and~$\ell$, in $\Ocal(K \cdot n^2)$ time, and it extracts a $(K, 4)$-short jungle from the $(130K + 2, 5K)$-degree tangle or the $(5K + 1, 5K)$-matching tangle that this algorithm may output, in $\Ocal(K^3 \cdot n^2)$ time~\cite[Lemmas 3.9 and 3.12]{FominP19}.

\begin{proposition}[{\textcite[proof of Theorem 9.9]{FominP19}}]\label{prop:semicomplete-jungle-algorithm}
  Let~$T$ be an $n$-vertex semicomplete digraph and let~$K$ be a positive integer.
  There is an algorithm that, in time \mbox{$\Ocal(K^3 \cdot n^2)$}, outputs a path decomposition of~$T$ of width at most $140K$ or a $(K, 4)$-short jungle of~$T$.
\end{proposition}

We apply \zcref{prop:semicomplete-jungle-algorithm} to the detour completion~$D^\star$ of~$D$ (\zcref{def:detours-and-completions}).
By \zcref{obs:out-degree-h-semicomplete}, an \hsemi{h} digraph~$D$ on~$n$ vertices has at most $hn/2$ pairs of non-neighbors.
Thus~$D^\star$ can be computed in $\Ocal((h+1) \cdot n^2)$ time by computing $|\detours{D}{u}{v}|$ and $|\detours{D}{v}{u}|$ for each such pair.

\subsection{From a short jungle to a \titlemath{k}-triple}\label{sec:jungle-to-triple}

We now show how to obtain a $k$-triple of~$D$ from a short jungle of~$D^\star$.
We first find in the jungle a set~$M$ of $8k$ vertices and two large sets~$Z^-$ and~$Z^+$ such that~$Z^-$ is complete to~$M$ and~$M$ is complete to~$Z^+$ in~$D$.
Then we build the $k$-triple from arcs of the paths of the jungle from~$Z^+$ to~$Z^-$, and we use detours (\zcref{def:detours-and-completions}) when such an arc is not an arc of~$D$.

For a digraph~$T$ and $X \subseteq V(T)$, we denote by $\Gamma^-_T(X)$ the set of vertices~$v \in V(T)$ that are complete to~$X$, and by $\Gamma^+_T(X)$ the set of vertices~$v$ such that~$X$ is complete to~$v$ (\zcref{fig:common-neighbors}).
Both sets are disjoint from~$X$, since~$T$ has no loops.

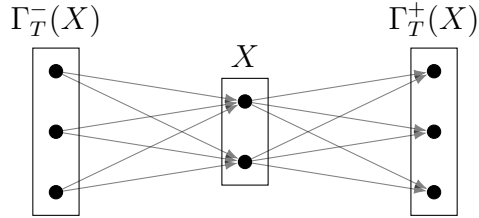
\begin{figure}[ht]
  \centering
  \begin{tikzpicture}[yscale=.8]
  \foreach \i/\y in {1/0, 2/1, 3/2} {
      \node[blackvertex] (l\i) at (0, \y) {};
      \node[blackvertex] (r\i) at (5, \y) {};
    }
  \foreach \j/\y in {1/.5, 2/1.5} {
      \node[blackvertex] (x\j) at (2.5, \y) {};
    }
  \foreach \i in {1, 2, 3} {
      \foreach \j in {1, 2} {
          \draw[arrow, opacity=.5] (l\i) -- (x\j);
          \draw[arrow, opacity=.5] (x\j) -- (r\i);
        }
    }
  \node[rectangle, draw, inner sep=6pt, fit=(l1)(l3), label=90:{$\Gamma^-_T(X)$}] {};
  \node[rectangle, draw, inner sep=6pt, fit=(x1)(x2), label=90:{$X$}] {};
  \node[rectangle, draw, inner sep=6pt, fit=(r1)(r3), label=90:{$\Gamma^+_T(X)$}] {};
\end{tikzpicture}%
%
  \caption{A set~$X$ of vertices of a digraph~$T$ with the sets~$\Gamma^-_T(X)$ and $\Gamma^+_T(X)$. Every vertex of $\Gamma^-_T(X)$ is complete to~$X$, and~$X$ is complete to every vertex of $\Gamma^+_T(X)$. The other arcs are omitted.}
  \label{fig:common-neighbors}
\end{figure}

The proof of the following lemma is a double count.
In a semicomplete digraph~$T$, we count, for each set~$M$ of a fixed size, the pairs~$(a, b)$ of distinct vertices with~$a$ complete to~$M$ and~$M$ complete to~$b$, which are exactly the pairs with $M \subseteq \detours{T}{a}{b}$.
The paths of length two of~$T$ make this number large on average over all sets~$M$, and a set~$M$ with many such pairs has many common in-neighbors and many common out-neighbors.
We then choose the vertices of~$M$ one at a time, each time keeping the average over the sets that contain the chosen vertices as large as possible, which gives, in polynomial time, a set~$M$ whose number of pairs is at least the average.
This is the method of conditional expectations~\cite{AlonS16}, written as a double count, without probability.

\begin{lemma}\label{lem:sandwich}
  Let~$L$ be a positive integer and~$T$ be a semicomplete digraph on $n \geq 36L$ vertices.
  There is an algorithm with run time~$\Ocal(L^2 \cdot n^3)$ that, given~$T$, outputs a set~$M \subseteq V(T)$ of size~$L$ such that $|\Gamma^-_T(M)| \geq n/8^L$ and $|\Gamma^+_T(M)| \geq n/8^L$.
\end{lemma}

\begin{proof}
  We first write the counts as sums over pairs of vertices.
  For distinct vertices~$a$ and~$b$ of~$T$, let $\delta(a, b) = |\detours{T}{a}{b}|$ be the number of detours from~$a$ to~$b$.
  For a set~$S \subseteq V(T)$ of at most~$L$ vertices, let $\Pi(S)$ be the set of ordered pairs~$(a, b)$ of distinct vertices with $S \subseteq \detours{T}{a}{b}$.
  Let $\Phi(S)$ be the number of pairs~$(M', (a, b))$ such that~$M'$ is a set of~$L$ vertices with $S \subseteq M'$, $(a, b)$ is an ordered pair of distinct vertices, and $M' \subseteq \detours{T}{a}{b}$.
  Note that, for a set~$X \subseteq V(T)$ and distinct vertices~$a$ and~$b$, we have $X \subseteq \detours{T}{a}{b}$ if and only if~$a$ is complete to~$X$ and~$X$ is complete to~$b$, that is, $a \in \Gamma^-_T(X)$ and $b \in \Gamma^+_T(X)$.
  Every pair counted by $\Phi(S)$ has $(a, b) \in \Pi(S)$, as $S \subseteq M'$.
  Conversely, for $(a, b) \in \Pi(S)$, the sets~$M'$ with $S \subseteq M' \subseteq \detours{T}{a}{b}$ are the unions of~$S$ with a subset of $\detours{T}{a}{b} \setminus S$ of size~$L - |S|$, so
  \begin{equation*}
    \Phi(S) = \sum_{(a, b) \in \Pi(S)} \binom{\delta(a, b) - |S|}{L - |S|} \enspace.
  \end{equation*}

  Next, we show that, if $|S| < L$, then some vertex~$x \notin S$ satisfies $\Phi(S \cup \{x\}) \geq \frac{L - |S|}{n - |S|} \cdot \Phi(S)$, that is, the average of $|\Pi(M')|$ over the sets~$M'$ of~$L$ vertices that contain $S \cup \{x\}$ is at least the average over those that contain~$S$.
  Every pair~$(M', (a, b))$ counted by $\Phi(S)$ is counted by $\Phi(S \cup \{x\})$ exactly for the $L - |S|$ vertices~$x \in M' \setminus S$, so $\sum_{x \notin S} \Phi(S \cup \{x\}) = (L - |S|) \cdot \Phi(S)$.
  As this sum has $n - |S|$ terms, some term is at least their average, which is the claim.
  The algorithm starts with $S = \emptyset$ and, while $|S| < L$, adds to~$S$ a vertex~$x \notin S$ that maximizes $\Phi(S \cup \{x\})$.
  It outputs $M = S$ when $|S| = L$, and by the claim applied at each of the~$L$ steps,
  \begin{equation*}
    \Phi(M) \geq \Phi(\emptyset) \cdot \prod_{i = 0}^{L-1} \frac{L - i}{n - i} = \frac{\Phi(\emptyset)}{\binom{n}{L}} \enspace.
  \end{equation*}

  We now bound $\Phi(\emptyset) = \sum_{(a, b) \in \Pi(\emptyset)} \binom{\delta(a, b)}{L}$ from below, where $\Pi(\emptyset)$ consists of all $n(n-1)$ ordered pairs of distinct vertices.
  The sum of $\delta(a, b)$ over these pairs is the number of paths of length two in~$T$.
  Any three vertices of~$T$ span such a path, since choosing one arc between each two of them gives a tournament on three vertices, in which some vertex has an in-neighbor and an out-neighbor.
  Distinct sets of three vertices give distinct paths, so~$T$ has at least $\binom{n}{3} = n(n-1)(n-2)/6$ paths of length two.
  For $(a, b) \in \Pi(\emptyset)$, let $q(a, b) = \max\{\delta(a, b) - L, 0\}$.
  We claim that $\binom{\delta(a, b)}{L} \geq \binom{n}{L} \cdot {(q(a, b)/n)}^L$.
  This is trivial if $q(a, b) = 0$.
  Otherwise $\delta(a, b) > L$, and for every $i \in \{0, \ldots, L-1\}$ we have $(\delta(a, b) - i)/(n - i) \geq (\delta(a, b) - L)/n$, as this is equivalent to $L(n - i) \geq i(n - \delta(a, b))$, which holds since $L \geq i$ and $n - i \geq n - \delta(a, b) \geq 0$.
  The product of these~$L$ inequalities gives
  \begin{equation*}
    \binom{\delta(a, b)}{L} \Big/ \binom{n}{L} = \prod_{i = 0}^{L-1} \frac{\delta(a, b) - i}{n - i} \geq {\left(\frac{\delta(a, b) - L}{n}\right)}^L = {\left(\frac{q(a, b)}{n}\right)}^L,
  \end{equation*}
  as claimed.
  Since $q(a, b) \geq \delta(a, b) - L$, the sum of $q(a, b)$ over the $n(n-1)$ pairs is at least $n(n-1)(n-2)/6 - L \cdot n(n-1)$, so their average is at least $(n - 2)/6 - L \geq n/8$, where the inequality is equivalent to $n \geq 24L + 8$ and holds as $n \geq 36L$.
  By the power mean inequality, $\sum_{(a, b) \in \Pi(\emptyset)} {q(a, b)}^L \geq n(n-1) \cdot {(n/8)}^L$, as $L \geq 1$.
  Thus $\Phi(\emptyset) \geq \binom{n}{L} \cdot n(n-1)/8^L$ by the bound on $\binom{\delta(a, b)}{L}$, and the bound on $\Phi(M)$ gives $\Phi(M) \geq n(n-1)/8^L$.

  Finally, we bound~$|\Gamma^-_T(M)|$ and $|\Gamma^+_T(M)|$.
  Since $|M| = L$, the pairs counted by $\Phi(M)$ are the pairs~$(M', (a, b))$ with $M' = M$ and $(a, b) \in \Pi(M)$, so $\Phi(M) = |\Pi(M)|$.
  Every pair of $\Pi(M)$ has its first vertex in $\Gamma^-_T(M)$, and each such vertex is the first vertex of at most $n - 1$ pairs.
  Thus $n(n-1)/8^L \leq \Phi(M) \leq |\Gamma^-_T(M)| \cdot (n - 1)$, which gives $|\Gamma^-_T(M)| \geq n/8^L$, and symmetrically $|\Gamma^+_T(M)| \geq n/8^L$.

  For the run time, we compute $\delta(a, b)$ for all pairs in $\Ocal(n^3)$ time, and the binomial coefficients~$\binom{c}{j}$ with $c \leq n$ and $j \leq L$ by Pascal's rule, in $\Ocal(Ln)$ additions.
  We maintain the sets~$\Gamma^-_T(S)$ and $\Gamma^+_T(S)$, so that $(a, b) \in \Pi(S \cup \{x\})$ is tested in constant time, as it means $(a, b) \in \Pi(S)$ and $(a, x), (x, b) \in E(T)$.
  Thus, in each of the~$L$ steps, each of the at most~$n$ values~$\Phi(S \cup \{x\})$ is computed by the formula above as a sum of at most~$n^2$ binomial coefficients, so each step takes $\Ocal(n^3)$ additions and the~$L$ steps take $\Ocal(L \cdot n^3)$ additions.
  All numbers are at most $n^2 \cdot \binom{n}{L} \leq n^{L+2}$, as $L \leq n/2$, so each addition takes $\Ocal(L)$ time on words of $\Ocal(\log n)$ bits, the total run time is $\Ocal(L^2 \cdot n^3)$, and the result follows.
\end{proof}

Next, let $(p_0, \ldots, p_r)$ be a short path from~$Z^+$ to~$Z^-$.
Every vertex~$x \in M$ is an in-neighbor of~$p_0$ and an out-neighbor of~$p_r$, since~$M$ is complete to~$Z^+$ and~$Z^-$ is complete to~$M$.
We say that~$x$ \emph{flips} at the arc~$(p_{i-1}, p_i)$, where $i \in [r]$, if~$x$ is an in-neighbor of~$p_{i-1}$ and an out-neighbor of~$p_i$.
We show that many vertices of~$M$ flip at the same arc.
The following lemma does not use the arcs of the path, so we state it for a sequence of vertices.
See \zcref{fig:triples-flip}.

\begin{lemma}\label{lem:flip}
  Let~$D$ be a digraph, $M \subseteq V(D)$, and $p_0, \ldots, p_r$ be vertices of~$D$ with $r \geq 1$.
  Assume that~$M$ is complete to~$p_0$,~$p_r$ is complete to~$M$, and every~$p_i$ with $0 < i < r$ is outside~$M$ and adjacent to every vertex of~$M$.
  Then there are an index~$i \in [r]$ and a set~$F \subseteq M$ of size at least $|M|/r$ such that~$p_i$ is complete to~$F$ and~$F$ is complete to~$p_{i-1}$.
\end{lemma}

\begin{figure}[ht]
  \centering
  \begin{tikzpicture}
  \foreach \i/\x in {0/0, 1/1.5, 2/3, 3/4.5, 4/6} {
      \node[blackvertex, label={[yshift=-.1cm, name=lp\i]-90:{$p_\i$}}] (p\i) at (\x, 0) {};
    }
  \draw[arrow] (p0) -- (p1);
  \draw[arrow] (p1) -- (p2);
  \draw[arrow, very thick] (p2) -- (p3);
  \draw[arrow] (p3) -- (p4);
  \node[blackvertex, label={[xshift=.1cm]0:{$x$}}] (t) at (3, 2.6) {};
  \node[vertex] (t1) at (1.8, 2.6) {};
  \node[vertex] (t2) at (4.2, 2.6) {};
  \draw[arrow] (t) -- (p0);
  \draw[arrow] (t) -- (p1);
  \draw[arrow] (t) -- (p2);
  \draw[arrow] (p3) -- (t);
  \draw[arrow] (p4) -- (t);
  \node[rectangle, draw, inner sep=6pt, fit=(t1)(t2), label=90:{$M$}] {};
  \node[rectangle, draw, inner sep=4pt, fit=(p0)(lp0), label=180:{$Z^+$}] {};
  \node[rectangle, draw, inner sep=4pt, fit=(p4)(lp4), label=0:{$Z^-$}] {};
\end{tikzpicture}%
%
  \caption{A path~$P = (p_0, \ldots, p_4)$ from~$Z^+$ to~$Z^-$ and a vertex~$x \in M$ with $\mu(x) = 3$ in the proof of \zcref{lem:flip}:~$x$ is an in-neighbor of~$p_0$,~$p_1$, and~$p_2$ and an out-neighbor of~$p_3$ and~$p_4$, so it flips at the arc~$(p_2, p_3)$, drawn thick. If \zcref{lem:flip} gives the index~$i = 3$, then $x \in F$. Arcs incident to the other vertices of~$M$ are omitted.}
  \label{fig:triples-flip}
\end{figure}
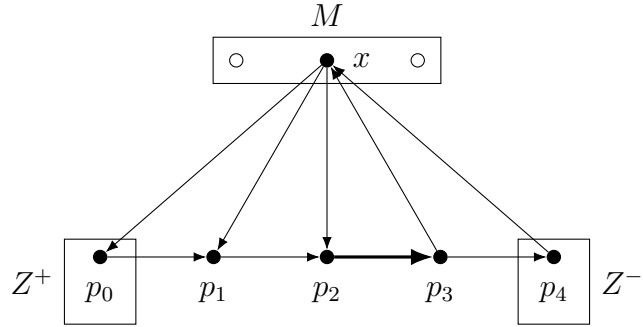

\begin{proof}
  For $x \in M$, let $\mu(x)$ be the smallest $i \in [r]$ with $(p_i, x) \in E(D)$, which is well defined since $(p_r, x) \in E(D)$.
  Then $(p_{\mu(x)}, x) \in E(D)$.
  If $\mu(x) = 1$, then $(x, p_0) \in E(D)$ since~$M$ is complete to~$p_0$.
  If $\mu(x) \geq 2$, then~$p_{\mu(x) - 1}$ is outside~$M$ and adjacent to~$x$ by hypothesis, and $(p_{\mu(x) - 1}, x) \notin E(D)$ by the minimality of $\mu(x)$, so $(x, p_{\mu(x) - 1}) \in E(D)$.
  Now, the sets~$F_i = \{x \in M \mid \mu(x) = i\}$ with $i \in [r]$ partition~$M$, so some~$F_i$ has size at least $|M|/r$, and the result follows with $F = F_i$.
\end{proof}

The next lemma extracts a $k$-triple of~$D$ from a short jungle of any completion~$S$ of~$D$, and we apply it to the detour completion~$D^\star$ in this section and, in \zcref{sec:anti-triples}, to the completion that adds both arcs between every two non-neighbors.
Its second output arises only when, on many paths of the jungle, the arc at which many vertices of~$M$ flip is not an arc of~$D$.

\begin{lemma}\label{lem:jungle-extraction}
  Let~$k$ and~$s$ be positive integers,~$h$ be a non-negative integer, and $\fref{lem:jungle-extraction}(k, h, s) = (h+1) \cdot k \cdot 2^{27k+1} + 2^{8k} \cdot s$.
  Let~$D$ be an \hsemi{h} digraph,~$S$ be a completion of~$D$, and~$Z$ be a $(J, 4)$-short jungle of~$S$, where $J = \fref{lem:jungle-extraction}(k, h, s)$.
  There is an algorithm with run time~$\Ocal(k^2 \cdot J^3)$ that, given~$D$ and~$Z$, outputs one of the following:
  \begin{outline}[enumerate]
    \1 a $k$-triple of~$D$, or
    \1 a set~$B \subseteq V(D)$ of size~$2k$ and~$s$ pairs~$(a_1, c_1), \ldots, (a_s, c_s)$ of non-neighbors of~$D$ whose $2s$ vertices are pairwise distinct, and outside~$B$.
       For every $j \in [s]$, we have $(c_j, a_j) \in E(S)$ and, in~$D$,~$a_j$ is complete to~$B$ and~$B$ is complete to~$c_j$.
  \end{outline}
\end{lemma}

\begin{proof}
  Let $L = 8k$, $m = k \cdot 8^k$, $n_0 = 2m \cdot 8^L = k \cdot 2^{27k+1}$, and $m' = 2^L \cdot s$, so that $J = (h+1) \cdot n_0 + m'$.
  Discarding vertices of~$Z$ if necessary, we assume that $|Z| = J$.

  First, we find~$M$,~$Z^-$, and~$Z^+$.
  Since $D[Z]$ is \hsemi{h} and $J \geq (h+1) \cdot n_0$, we apply \zcref{lem:semicomplete-from-h-semicomplete} to $D[Z]$ to find $Z' \subseteq Z$ with $|Z'| = n_0$ such that $D[Z']$ is semicomplete.
  As $L \geq 8$, we have $n_0 \geq 8^L \geq 36L$.
  Thus the algorithm of \zcref{lem:sandwich} applied to $D[Z']$ gives a set~$M \subseteq Z'$ of size~$L$ such that $\Gamma^-_{D[Z']}(M)$ and $\Gamma^+_{D[Z']}(M)$ both have size at least $n_0/8^L = 2m$.
  Let~$Z^-$ be a subset of $\Gamma^-_{D[Z']}(M)$ of size~$m$ and~$Z^+$ be a subset of $\Gamma^+_{D[Z']}(M) \setminus Z^-$ of size~$m$.
  Then~$M$,~$Z^-$, and~$Z^+$ are pairwise disjoint,~$Z^-$ is complete to~$M$ in~$D$, and~$M$ is complete to~$Z^+$ in~$D$.
  We write $Z^+ = \{v_1, \ldots, v_m\}$ and $Z^- = \{w_1, \ldots, w_m\}$, and, for $j \in [m]$, we let~$\mathcal{P}_j$ be the given family of~$J$ internally disjoint paths of~$S$ from~$v_j$ to~$w_j$ of length at most~$4$.
  Since these paths are internally disjoint, every vertex other than~$v_j$ and~$w_j$ lies on at most one path of~$\mathcal{P}_j$, as an internal vertex.

  Now, let $j \in [m]$.
  We say that a path of~$\mathcal{P}_j$ is \emph{good} if none of its internal vertices belongs to~$M$ and each of them is adjacent in~$D$ to every vertex of~$M$.
  At most $L(h+1)$ vertices belong to~$M$ or are non-neighbors of some vertex of~$M$, since each of the~$L$ vertices of~$M$ has at most~$h$ non-neighbors.
  Each of them is an internal vertex of at most one path of~$\mathcal{P}_j$, so at most $L(h+1)$ paths of~$\mathcal{P}_j$ are not good.
  Let $P = (p_0, \ldots, p_r)$ be a good path of~$\mathcal{P}_j$.
  Since~$P$ is good,~$M$ is complete to~$v_j$,~$w_j$ is complete to~$M$ in~$D$, and $r \geq 1$ as $v_j \neq w_j$, we can apply \zcref{lem:flip} to~$D$,~$M$, and $p_0, \ldots, p_r$.
  Since $r \leq 4$, it gives an index~$i_P \in [r]$ and a set~$F_P \subseteq M$ of size at least $L/4 = 2k$, which we fix.
  Let $a_P = p_{i_P}$ and $c_P = p_{i_P - 1}$.
  Then $(c_P, a_P)$ is an arc of~$P$ and of~$S$, and in~$D$ the vertex~$a_P$ is complete to~$F_P$ and~$F_P$ is complete to~$c_P$.

  Next, we show that at most $2h$ good paths~$P$ of~$\mathcal{P}_j$ with $(c_P, a_P) \notin E(D)$ have $c_P = v_j$ or $a_P = w_j$.
  For such a path,~$c_P$ and~$a_P$ are non-neighbors of~$D$, since~$S$ is a completion of~$D$.
  If $c_P = v_j$, then $a_P \neq w_j$, as~$v_j$ and~$w_j$ are adjacent in the semicomplete digraph~$D[Z']$, so~$a_P$ is an internal vertex of~$P$ and a non-neighbor of~$v_j$.
  Since~$v_j$ has at most~$h$ non-neighbors, at most~$h$ of these paths have $c_P = v_j$, and symmetrically at most~$h$ of them have $a_P = w_j$.
  For every other such path, both~$c_P$ and~$a_P$ are internal vertices of~$P$.

  Assume first that, for some $j \in [m]$, at least~$m'$ good paths~$P$ of~$\mathcal{P}_j$ have \mbox{$(c_P, a_P) \notin E(D)$} with~$c_P$ and~$a_P$ internal vertices of~$P$.
  For each of them, fix a subset~$B_P$ of~$F_P$ of size~$2k$.
  Since~$M$ has $\binom{8k}{2k} \leq 2^L$ subsets of size~$2k$ and $m' = 2^L \cdot s$, there is a set~$B \subseteq M$ with $B_P = B$ for at least~$s$ of these paths~$P$.
  We output~$B$ and the pairs~$(a_P, c_P)$ of~$s$ such paths.
  Their $2s$ vertices are pairwise distinct, since $a_P \neq c_P$ and each of them is an internal vertex of only one path of~$\mathcal{P}_j$, and they are outside $M \supseteq B$, since these paths are good.
  As $B \subseteq F_P$, the vertex~$a_P$ is complete to~$B$ and~$B$ is complete to~$c_P$.

  Assume now that this is not the case, and let $j \in [m]$.
  Then~$\mathcal{P}_j$ contains at most $L(h+1)$ paths that are not good and at most $2h + m' - 1$ good paths~$P$ with $(c_P, a_P) \notin E(D)$, so it contains at least
  \begin{equation*}
    J - L(h+1) - 2h - m' + 1 = (h+1)(n_0 - L) - 2h + 1 \geq 4m
  \end{equation*}
  good paths~$P$ with $(c_P, a_P) \in E(D)$, where the inequality uses $n_0 \geq 4m + L + 2$, which holds as $n_0 = 2m \cdot 8^L$ and $m \geq L \geq 8$.
  For $j = 1, \ldots, m$ in turn, we choose such a path~$P_j$ whose vertices~$a_{P_j}$ and~$c_{P_j}$ are outside
  \begin{equation*}
    X_j = (Z^+ \cup Z^- \cup \{a_{P_1}, c_{P_1}, \ldots, a_{P_{j-1}}, c_{P_{j-1}}\}) \setminus \{v_j, w_j\} \enspace.
  \end{equation*}
  A vertex of~$X_j$ differs from~$v_j$ and~$w_j$, so it lies on at most one path of~$\mathcal{P}_j$ and excludes at most one path.
  As $|X_j| \leq 2m - 2 + 2(j - 1) < 4m$, the path~$P_j$ exists.
  Now, the~$2m$ vertices~$a_{P_j}$ and~$c_{P_j}$ with $j \in [m]$ are pairwise distinct.
  To see this, note that $a_{P_j} \neq c_{P_j}$, as they are consecutive on~$P_j$, and let $j < j'$.
  The vertices~$a_{P_j}$ and~$c_{P_j}$ are outside $(Z^+ \cup Z^-) \setminus \{v_j, w_j\}$, so they differ from~$v_{j'}$ and~$w_{j'}$ and belong to~$X_{j'}$.
  By the choice of~$P_{j'}$, they differ from~$a_{P_{j'}}$ and~$c_{P_{j'}}$.
  All $2m$ vertices are outside~$M$, since each of them is an internal vertex of a good path or belongs to $Z^+ \cup Z^-$.

  The following claim gives~$k$ of the paths~$P_j$ whose sets~$F_{P_j}$ share~$k$ vertices.

  \begin{claim}\label{clm:common-subset}
    Let $F_1, \ldots, F_m$ be subsets of~$M$ of size at least $2k$.
    Then there are sets~$B \subseteq M$ and $I \subseteq [m]$ of size~$k$ such that $B \subseteq F_j$ for every $j \in I$.
  \end{claim}

  \begin{cproof}
    Each set~$F_j$ has at least $\binom{2k}{k}$ subsets of size~$k$, and~$M$ has $\binom{8k}{k}$ subsets of size~$k$.
    Thus, some set~$B \subseteq M$ of size~$k$ satisfies $B \subseteq F_j$ for at least $m \cdot \binom{2k}{k} / \binom{8k}{k}$ indices~$j \in [m]$.
    Since $2k - i > k$ and $8k - i \leq 8k$ for every $i \in \{0, \ldots, k-1\}$, we have
    \begin{equation*}
      \binom{2k}{k} \Big/ \binom{8k}{k} = \prod_{i = 0}^{k-1} \frac{2k - i}{8k - i} > \prod_{i = 0}^{k-1} \frac{k}{8k} = 8^{-k} \enspace.
    \end{equation*}
    Thus, $B \subseteq F_j$ for at least $m \cdot 8^{-k} = k$ indices~$j \in [m]$, and the claim follows.
  \end{cproof}

  By \zcref{clm:common-subset} applied to $F_{P_1}, \ldots, F_{P_m}$, there are sets~$B \subseteq M$ and $I \subseteq [m]$ of size~$k$ such that $B \subseteq F_{P_j}$ for every $j \in I$.
  For $j \in I$, the vertex~$a_{P_j}$ is complete to~$B$ and~$B$ is complete to~$c_{P_j}$, and $\{c_{P_j} \mid j \in I\}$ is matched to $\{a_{P_j} \mid j \in I\}$ by the arcs~$(c_{P_j}, a_{P_j})$ of~$D$.
  Thus $(\{a_{P_j} \mid j \in I\}, B, \{c_{P_j} \mid j \in I\})$ is a $k$-triple of~$D$, which we output.
  See \zcref{fig:triples-extraction}.

  The run time follows from the bounds below, where $m' \leq J$ and $m^2 \leq m \cdot L \cdot 2^L \leq n_0 \leq J$.
  By \zcref{lem:semicomplete-from-h-semicomplete}, we find~$Z'$ in $\Ocal(n_0 \cdot J)$ time.
  By \zcref{lem:sandwich}, we find~$M$ in $\Ocal(L^2 \cdot n_0^3) = \Ocal(k^2 \cdot J^3)$ time.
  For the~$m$ families~$\mathcal{P}_j$, the good paths, the indices~$i_P$, and the sets~$F_P$ are computed in $\Ocal(mJL)$ time.
  In the first case, the set~$B$ is found in $\Ocal(m' \cdot L \cdot 2^L) = \Ocal(J^2)$ time by checking all subsets of~$M$ of size~$2k$.
  In the second case, the paths~$P_j$ are chosen in $\Ocal(m^2 J)$ time by comparing~$a_P$ and~$c_P$ with the at most $4m$ vertices of~$X_j$.
  The sets~$B$ and~$I$ are found in $\Ocal(m \cdot L \cdot 2^L)$ time by checking all subsets of~$M$ of size~$k$.
  The total run time is $\Ocal(k^2 \cdot J^3)$, and the result follows.
\end{proof}

\begin{figure}[ht]
  \centering
  \begin{tikzpicture}[yscale=.6]
  \foreach \j/\y in {1/1, 2/2} {
      \node[blackvertex, label={[xshift=-.1cm]180:{$b_\j$}}] (b\j) at (-2, \y) {};
    }
  \node[blackvertex, label={[yshift=.1cm]90:{$v_1$}}] (v1) at (0, 6.2) {};
  \node[blackvertex, label={[xshift=-.05cm, yshift=.05cm]135:{$c_{P_1}$}}] (c1) at (0, 3.6) {};
  \node[blackvertex, label={[xshift=-.05cm, yshift=-.05cm]-135:{$a_{P_1}$}}] (a1) at (0, -.6) {};
  \node[blackvertex, label={[yshift=-.1cm]-90:{$w_1$}}] (w1) at (0, -3.2) {};
  \node[blackvertex, label={[yshift=.1cm]90:{$v_2$}}] (v2) at (1.8, 3.6) {};
  \node[blackvertex, label={[xshift=.05cm, yshift=-.05cm]-45:{$a_{P_2}$}}] (a2) at (1.8, -.6) {};
  \node[blackvertex, label={[yshift=-.1cm]-90:{$w_2$}}] (w2) at (1.8, -3.2) {};
  \draw[arrow] (v1) -- (c1);
  \draw[arrow, very thick] (c1) -- (a1);
  \draw[arrow] (a1) -- (w1);
  \draw[arrow, very thick] (v2) -- (a2);
  \draw[arrow] (a2) -- (w2);
  \foreach \j in {1, 2} {
      \draw[arrow, opacity=.5] (a1) -- (b\j);
      \draw[arrow, opacity=.5] (a2) -- (b\j);
      \draw[arrow, opacity=.5] (b\j) -- (c1);
      \draw[arrow, opacity=.5] (b\j) -- (v2);
    }
  \node[rectangle, draw, fit=(b1)(b2), label=90:{$B$}] {};
  \node[rectangle, draw, fit=(c1)(v2), label=0:{$C$}] {};
  \node[rectangle, draw, fit=(a1)(a2), label=-90:{$A$}] {};
\end{tikzpicture}%
%
  \caption{The $2$-triple~$(\{a_{P_1}, a_{P_2}\}, B, \{c_{P_1}, v_2\})$ built in the second case of the proof of \zcref{lem:jungle-extraction} from the paths~$P_1 = (v_1, c_{P_1}, a_{P_1}, w_1)$ and $P_2 = (v_2, a_{P_2}, w_2)$, which go from $v_1, v_2 \in Z^+$ to $w_1, w_2 \in Z^-$. Here $i_{P_1} = 2$ and $i_{P_2} = 1$, so $c_{P_2} = v_2$. The set~$B = \{b_1, b_2\}$ is a subset of $F_{P_1} \cap F_{P_2}$. Thick arcs form the matching, thin black arcs are the other arcs of~$P_1$ and~$P_2$, gray arcs join~$A$ to~$B$ and~$B$ to~$C$, and all other arcs are omitted.}
  \label{fig:triples-extraction}
\end{figure}
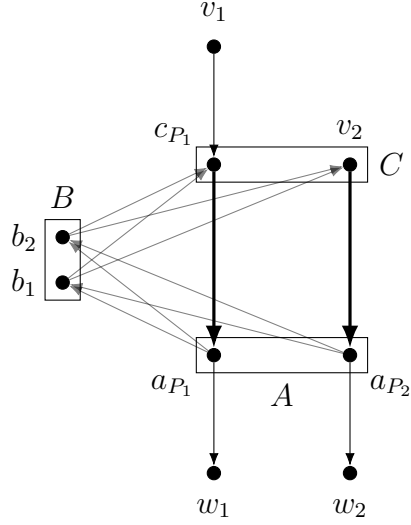

In the second output of \zcref{lem:jungle-extraction} with $S = D^\star$, every arc~$(c_j, a_j)$ joins non-neighbors of~$D$ and is an arc of~$D^\star$, so $|\detours{D}{c_j}{a_j}| \geq |\detours{D}{a_j}{c_j}|$ by \zcref{def:detours-and-completions}.
The next lemma uses the detours from~$c_j$ to~$a_j$ to replace~$a_j$ or~$c_j$ in a triple whose middle set is half of~$B$.
Either the arcs from the vertices~$c_j$ to the vertices~$a_i$ contain the matching of a $k$-triple, or by Kőnig's Theorem (\zcref{thm:konig}) fewer than~$k$ vertices meet all of them.
In the second case, we show that many of the remaining pairs have many detours from~$c_j$ to~$a_j$.
If not, then by \zcref{obs:out-degree-h-semicomplete} some~$a_j$ of a pair with few detours has many out-neighbors~$a_i$ among these pairs.
All but at most~$h$ of them lie in $\detours{D}{a_j}{c_j}$, as no arc goes from~$c_j$ to them, so the hypothesis gives at least as many detours from~$c_j$ to~$a_j$, a contradiction.
See \zcref{fig:triples-detour-pairs}.

\begin{lemma}\label{lem:detour-pairs}
  Let~$k$ be a positive integer,~$h$ be a non-negative integer, and $\fref{lem:detour-pairs}(k, h) = 3k \cdot \binom{2k}{k} + (4k + 3)h + k$.
  Let~$D$ be an \hsemi{h} digraph on~$n$ vertices and $B \subseteq V(D)$ be a set of size~$2k$.
  Let $(a_1, c_1), \ldots, (a_s, c_s)$ be pairs of vertices of~$D$, with $s \geq \fref{lem:detour-pairs}(k, h)$, whose $2s$ vertices are distinct and outside~$B$.
  Assume that, for every $j \in [s]$, the vertex~$a_j$ is complete to~$B$,~$B$ is complete to~$c_j$, and $|\detours{D}{c_j}{a_j}| \geq |\detours{D}{a_j}{c_j}|$.
  Then~$D$ contains a $k$-triple, which can be found in $\Ocal(k \cdot s^2 + sn)$ time.
\end{lemma}

\begin{figure}[ht]
  \centering
  \begin{tikzpicture}[yscale=.6]
  \foreach \j/\y in {1/1, 2/2} {
      \node[blackvertex, label={[xshift=-.1cm]180:{$b_\j$}}] (b\j) at (-2, \y) {};
    }
  \node[blackvertex, label={[yshift=.1cm]90:{$c_1$}}] (c1) at (0, 3.6) {};
  \node[blackvertex, label={[xshift=-.05cm, yshift=.05cm]135:{$u_2$}}] (u2) at (1.2, 3.6) {};
  \node[blackvertex, label={[xshift=.05cm, yshift=-.05cm]-45:{$u_1$}}] (u1) at (0, -.6) {};
  \node[blackvertex, label={[yshift=-.1cm]-90:{$a_2$}}] (a2) at (1.2, -.6) {};
  \node[vertex, faded, label={[yshift=-.1cm, gray]-90:{$a_1$}}] (a1) at (-1.2, -2.4) {};
  \node[vertex, faded, label={[yshift=.1cm, gray]90:{$c_2$}}] (c2) at (2.4, 5.4) {};
  \foreach \j in {1, 2} {
      \draw[arrow, opacity=.5] (u1) -- (b\j);
      \draw[arrow, opacity=.5] (a2) -- (b\j);
      \draw[arrow, opacity=.5] (b\j) -- (c1);
      \draw[arrow, opacity=.5] (b\j) -- (u2);
    }
  \draw[arrow, very thick] (c1) -- (u1);
  \draw[arrow, very thick] (u2) -- (a2);
  \draw[arrow] (u1) -- (a1);
  \draw[arrow] (c2) -- (u2);
  \node[rectangle, draw, fit=(b1)(b2), label=90:{$B'$}] {};
  \node[rectangle, draw, fit=(c1)(u2), label=0:{$C$}] {};
  \node[rectangle, draw, fit=(u1)(a2), label=0:{$A$}] {};
\end{tikzpicture}%
%
  \caption{Building a $2$-triple in the proof of \zcref{lem:detour-pairs}. The vertex~$u_1 \in \detours{D}{c_1}{a_1}$ is complete to~$B'$ and replaces~$a_1$, and~$B'$ is complete to $u_2 \in \detours{D}{c_2}{a_2}$, which replaces~$c_2$. Thick arcs form the matching, and the arcs between~$B'$ and the replaced vertices~$a_1$ and~$c_2$ are omitted.}
  \label{fig:triples-detour-pairs}
\end{figure}
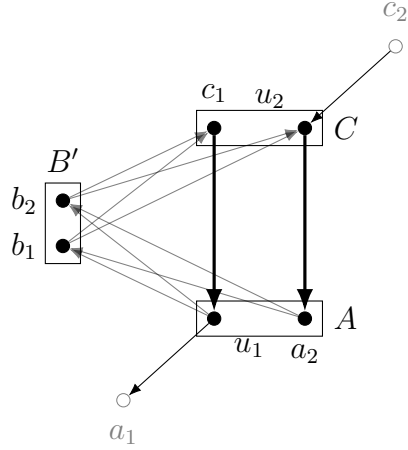

\begin{proof}
  For $j \in [s]$, let $Q_j = \detours{D}{c_j}{a_j}$, so that $(c_j, u), (u, a_j) \in E(D)$ for every $u \in Q_j$.
  Let $R = k \cdot \binom{2k}{k}$ and $q = 2kh + 2k + R$.

  Let~$G$ be the bipartite graph with bipartition~$(\{a_1, \ldots, a_s\}, \{c_1, \ldots, c_s\})$ and an edge between~$a_i$ and~$c_j$ whenever $(c_j, a_i) \in E(D)$.
  If~$G$ has a matching of size~$k$, let~$A'$ and~$C'$ be the sets of its ends among the vertices~$a_i$ and among the vertices~$c_j$, respectively.
  Let $B' \subseteq B$ be a set of size~$k$.
  Then~$C'$ is matched to~$A'$ by the arcs~$(c_j, a_i)$ corresponding to the edges of the matching, and $(A', B', C')$ is a $k$-triple of~$D$.
  Otherwise, by \zcref{thm:konig},~$G$ has a vertex cover~$W$ of size at most $k - 1$.
  Let~$I$ be the set of indices~$j \in [s]$ with $a_j, c_j \notin W$.
  Each vertex of~$W$ is~$a_j$ or~$c_j$ for exactly one index~$j$, so $|I| \geq s - k + 1$.
  For all $i, j \in I$, we have $(c_j, a_i) \notin E(D)$, as otherwise~$W$ would not cover the edge between~$a_i$ and~$c_j$.
  Thus, for all $i, j \in I$, the vertex~$a_i$ is not in~$Q_j$, and~$c_i$ is not in~$Q_j$ since $(c_i, a_j) \notin E(D)$.

  Now, let~$I'$ be the set of indices~$j \in I$ with $|Q_j| \geq q$.
  We claim that $|I \setminus I'| \leq 2q - 4k + 3h$.
  By contradiction, assume that $I'' = I \setminus I'$ has $N \geq 2q - 4k + 3h + 1$ indices.
  The digraph~$D[\{a_i \mid i \in I''\}]$ is \hsemi{h}, so by \zcref{obs:out-degree-h-semicomplete} some~$a_j$ with $j \in I''$ has at least $(N - 1 - h)/2$ out-neighbors~$a_i$ with $i \in I''$.
  For each such~$i$, we have $(c_j, a_i) \notin E(D)$, so if~$a_i$ and~$c_j$ are adjacent, then $(a_i, c_j) \in E(D)$ and $a_i \in \detours{D}{a_j}{c_j}$.
  Since~$c_j$ has at most~$h$ non-neighbors, at least $(N - 1 - h)/2 - h$ of these vertices~$a_i$ belong to $\detours{D}{a_j}{c_j}$.
  Also, $B \subseteq \detours{D}{a_j}{c_j}$, as~$a_j$ is complete to~$B$ and~$B$ is complete to~$c_j$, and~$B$ contains no vertex~$a_i$.
  Thus $|Q_j| \geq |\detours{D}{a_j}{c_j}| \geq 2k + (N - 1 - 3h)/2 \geq q$, so $j \in I'$, a contradiction.
  Therefore,
  \begin{equation*}
    |I'| \geq |I| - (2q - 4k + 3h) \geq s - k + 1 - 2q + 4k - 3h = s - 2R - (4k + 3)h - k + 1 \geq R + 1 \enspace.
  \end{equation*}

  Now, let $j_1, \ldots, j_R$ be distinct indices of~$I'$.
  At most $2kh$ vertices are non-neighbors of some vertex of~$B$, since each of the $2k$ vertices of~$B$ has at most~$h$ non-neighbors.
  Thus, for $t = 1, \ldots, R$ in turn, we can choose a vertex~$u_t \in Q_{j_t} \setminus B$ that is adjacent to every vertex of~$B$ and distinct from $u_1, \ldots, u_{t-1}$, since at most $2kh + 2k + R - 1 < q \leq |Q_{j_t}|$ vertices are excluded.
  As~$u_t$ is adjacent to all $2k$ vertices of~$B$, there is a set~$B'_t \subseteq B$ of size~$k$ such that~$u_t$ is complete to~$B'_t$ or~$B'_t$ is complete to~$u_t$.
  Since~$B$ has $\binom{2k}{k}$ subsets of size~$k$ and $R = k \cdot \binom{2k}{k}$, there are a set~$B' \subseteq B$ and a set~$Y \subseteq [R]$ of size~$k$ with $B'_t = B'$ for every $t \in Y$.
  For $t \in Y$, let $(\alpha_t, \gamma_t) = (u_t, c_{j_t})$ if~$u_t$ is complete to~$B'$, and $(\alpha_t, \gamma_t) = (a_{j_t}, u_t)$ otherwise, in which case~$B'$ is complete to~$u_t$.
  In both cases,~$\alpha_t$ is complete to~$B'$,~$B'$ is complete to~$\gamma_t$, and $(\gamma_t, \alpha_t) \in E(D)$ since $u_t \in Q_{j_t}$.
  The vertices~$u_t$ with $t \in Y$ are pairwise distinct, outside~$B$, and, as shown above, distinct from all vertices~$a_i$ and~$c_i$ with $i \in I$.
  Thus the $2k$ vertices~$\alpha_t$ and~$\gamma_t$ with $t \in Y$ are distinct and outside~$B'$, and we conclude that $(\{\alpha_t \mid t \in Y\}, B', \{\gamma_t \mid t \in Y\})$ is a $k$-triple of~$D$.

  The run time follows from the following bounds.
  As described in \zcref{sec:prelim}, the matching or the vertex cover of~$G$ is found in $\Ocal(k \cdot s^2)$ time.
  The sets~$Q_j$ with $j \in I$ and their sizes are computed in $\Ocal(sn)$ time, and we mark the vertices adjacent to every vertex of~$B$ in $\Ocal(kn)$ time.
  Then the vertices~$u_1, \ldots, u_R$ are found in $\Ocal(R \cdot n)$ time, and the sets~$B'_t$ and~$Y$ in $\Ocal(k \cdot R^2)$ time by comparing the sets~$B'_t$ pairwise.
  Since $k \leq R \leq s$, the total run time is $\Ocal(k \cdot s^2 + sn)$, and the result follows.
\end{proof}

We are ready to prove the main result of this section.

\begin{theorem}\label{thm:large-dpw-to-triples}
  Let~$k$ be a positive integer,~$h$ be a non-negative integer, and $\fref{thm:large-dpw-to-triples}(k, h) = (h+1) \cdot k \cdot 2^{27k + 9}$.
  Every \hsemi{h} digraph~$D$ with $\dpw(D) > \fref{thm:large-dpw-to-triples}(k, h)$ contains a $k$-triple.
  There is an algorithm with run time~$\min\{{(h+1)}^3 \cdot 2^{\Ocal(k)} \cdot n^2, \Ocal(n^5)\}$ that, given~$k$ and an \hsemi{h} digraph~$D$ on~$n$ vertices, outputs a path decomposition of~$D$ of width at most $\fref{thm:large-dpw-to-triples}(k, h)$ or a $k$-triple of~$D$.
\end{theorem}

\begin{proof}
  Let $s = \fref{lem:detour-pairs}(k, h)$, $J = \fref{lem:jungle-extraction}(k, h, s)$, and $K = \lfloor \fref{thm:large-dpw-to-triples}(k, h) / 140 \rfloor$.
  We first show that $J \leq K$.
  As $\binom{2k}{k} \leq 4^k$ and $(4k + 3)h + k \leq (h+1)(4k + 3)$, we have $s \leq (h+1)(3k \cdot 4^k + 4k + 3) \leq (h+1) \cdot k \cdot 2^{2k + 4}$, where the last inequality uses $4k + 3 \leq 13k \cdot 4^k$.
  Thus,
  \begin{align*}
    140J & = 140(h+1) \cdot k \cdot 2^{27k + 1} + 140 \cdot 2^{8k} \cdot s \\
          & \leq (h+1) \cdot k \cdot (280 \cdot 2^{27k} + 140 \cdot 2^{10k + 4}) \leq (h+1) \cdot k \cdot 2^{27k + 9} \enspace,
  \end{align*}
  where the last inequality uses $140 \cdot 2^{10k + 4} \leq 232 \cdot 2^{27k}$, which gives $J \leq K$.

  Given~$D$, we first compute in $\Ocal(n^2)$ time the largest number~$h' \leq h$ of non-neighbors of a vertex of~$D$, and we replace~$h$ by~$h'$, which only decreases the bounds of the statement.
  Then, if $n \leq \fref{thm:large-dpw-to-triples}(k, h) + 1$, we output the path decomposition~$(V(D))$ of~$D$ with a single bag, whose width is $n - 1 \leq \fref{thm:large-dpw-to-triples}(k, h)$.
  Otherwise, we compute the detour completion~$D^\star$ of~$D$ in $\Ocal((h+1) \cdot n^2)$ time, and we run the algorithm of \zcref{prop:semicomplete-jungle-algorithm} on~$D^\star$ with~$K$, which takes $\Ocal(K^3 \cdot n^2)$ time.
  If it outputs a path decomposition of~$D^\star$, then its width is at most $140K \leq \fref{thm:large-dpw-to-triples}(k, h)$, it is a path decomposition of~$D$ by \zcref{obs:dpw-subgraph}, and we output it.
  Otherwise, it outputs a $(K, 4)$-short jungle~$Z$ of~$D^\star$.
  Since $K \geq J$, the set~$Z$ with~$J$ of the given paths for every pair of its vertices is a $(J, 4)$-short jungle of~$D^\star$, and we apply the algorithm of \zcref{lem:jungle-extraction} to~$D$ and~$Z$ with $S = D^\star$.
  If it outputs a $k$-triple of~$D$, we output it.
  Otherwise, it outputs a set~$B$ of size~$2k$ and pairs~$(a_1, c_1), \ldots, (a_s, c_s)$ as in its second item.
  For every $j \in [s]$, the vertices~$a_j$ and~$c_j$ are non-neighbors in~$D$ and $(c_j, a_j)$ is an arc of~$D^\star$, so $|\detours{D}{c_j}{a_j}| \geq |\detours{D}{a_j}{c_j}|$ by the definition of~$D^\star$.
  Thus, \zcref{lem:detour-pairs} applies to~$D$,~$B$, and these pairs, and we output the $k$-triple of~$D$ that it finds.
  The first statement follows, since if $\dpw(D) > \fref{thm:large-dpw-to-triples}(k, h)$, then~$D$ has no path decomposition of width at most $\fref{thm:large-dpw-to-triples}(k, h)$, so the algorithm outputs a $k$-triple of~$D$.

  For the run time, the steps above take $\Ocal((h+1) \cdot n^2 + K^3 \cdot n^2 + k^2 \cdot J^3 + k \cdot s^2 + sn)$ time.
  Since $h + 1 \leq s \leq J \leq K$ and $K = (h+1) \cdot 2^{\Ocal(k)}$, this is ${(h+1)}^3 \cdot 2^{\Ocal(k)} \cdot n^2$.
  Moreover, these steps run only when $n > \fref{thm:large-dpw-to-triples}(k, h) + 1 > 140K$, in which case~$k$, $h + 1$,~$s$,~$J$, and~$K$ are all less than~$n$, so the run time is also $\Ocal(n^5)$, and the result follows.
\end{proof}

For every~$k$, every semicomplete digraph with a large enough jungle contains a $k$-triple, a result that \textcite[Section 1 and (2.6)]{FradkinS13} credit to Chudnovsky, Scott, and Seymour.
The analogue for \hsemi{h} digraphs follows directly.
Indeed, let~$k$ be a positive integer and~$h$ be a non-negative integer.
By \zcref{lem:jungle-dpw}, an \hsemi{h} digraph that contains a jungle of size~$\fref{thm:large-dpw-to-triples}(k, h) + 2$ has directed pathwidth greater than $\fref{thm:large-dpw-to-triples}(k, h)$, so it contains a $k$-triple by \zcref{thm:large-dpw-to-triples}, and thus we obtain the following.

\begin{corollary}\label{cor:jungles-to-triples}
  Let~$k$ be a positive integer,~$h$ be a non-negative integer, and $\fref{cor:jungles-to-triples}(k, h) = \fref{thm:large-dpw-to-triples}(k, h) + 2$.
  Every \hsemi{h} digraph that contains a jungle of size~$\fref{cor:jungles-to-triples}(k, h)$ contains a $k$-triple.
\end{corollary}

\subsection{Rooted Immersion}\label{sec:rooted-immersion}

Given a digraph~$H$ and a semicomplete digraph, \textsc{Rooted Immersion} asks for an injective map from $V(H)$ to the vertices of the semicomplete digraph that sends prescribed roots to prescribed roots, together with pairwise arc-disjoint paths joining the images of the ends of every arc of~$H$.
The proof of~\mbox{\textcite[Theorem~9.9]{FominP19}}, with \zcref{cor:jungles-to-triples} for $h = 0$ in place of their Lemma~3.6, gives an algorithm for this problem on a semicomplete digraph on~$n$ vertices with run time~$\tower_2(\Ocal(|H|^2)) \cdot n^3$, where $|H| = |V(H)| + |E(H)|$.
Their rule~\mbox{\cite[Theorem 9.1]{FominP19}} that deletes a vertex without changing the answer needs a $t$-triple with \mbox{$t = 80|H|^2 + 80|H| + 5$}, so with our bound the path decompositions have width~$2^{\Ocal(|H|^2)}$.
Their dynamic programming~\cite[Theorem 10.8]{FominP19} on a path decomposition of width~$p$ takes $2^{\Ocal(|H| \cdot \log p + p \cdot |H| \cdot (\log |H| + \log p))} \cdot n$ time.
The proof of \textcite[(2.6)]{FradkinS13}, which gives the lemma replaced above~\cite[Lemma 3.6]{FominP19}, finds a $t$-triple in every jungle of size~$\tower_2(\Ramsey(2t) + \log_2 12)$, where the Ramsey number~$\Ramsey(2t)$, the least~$N$ such that every graph on~$N$ vertices has $2t$ pairwise adjacent or $2t$ pairwise non-adjacent vertices, is at most~$2^{4t}$~\cite{ErdosS35}.
With this bound in place of ours, the path decompositions have width~$\tower_3(\Ocal(|H|^2))$, and the run time is $\tower_4(\Ocal(|H|^2)) \cdot n^3$.

\subsection{Lower bounds}\label{sec:triples-lower-bounds}

The following lemma gives lower bounds on every function that can replace $\fref{thm:large-dpw-to-triples}(k, h)$ in \zcref{thm:large-dpw-to-triples}.

\begin{lemma}\label{lem:large-dpw-to-triples-tight}
  Let~$k$ be a positive integer and $n = \lfloor 2^{2(k-1)/3} \rfloor$.
  There is a tournament~$T$ on~$n$ vertices with no $k$-triple and $\dpw(T) > n - 2\sqrt{3n} - 1$, and for every non-negative integer~$h$ there is an \hsemi{h} digraph with no $k$-triple and directed pathwidth at least~$h$.
\end{lemma}

\begin{proof}
  For the tournament, we count the~$2^{\binom{n}{2}}$ tournaments on a fixed set~$V$ of~$n$ vertices.
  Since $n \leq 2^{2(k-1)/3}$, we have $3k \cdot \log_2 n \leq 2k^2 - 2k$.
  There are at most~$n^{3k}$ choices of a set~$B$ of~$k$ vertices and two $k$-tuples~$(a_1, \ldots, a_k)$ and $(c_1, \ldots, c_k)$ of vertices.
  Fix such a choice.
  If its $3k$ vertices are distinct, then exactly~$2^{\binom{n}{2} - 2k^2 - k}$ tournaments~$T$ make $(\{a_1, \ldots, a_k\}, B, \{c_1, \ldots, c_k\})$ a $k$-triple with $(c_i, a_i) \in E(T)$ for every $i \in [k]$, as this prescribes the arcs between $2k^2 + k$ pairs of vertices.
  Otherwise, no tournament does.
  Every $k$-triple of a tournament arises from such a choice, so at most $n^{3k} \cdot 2^{\binom{n}{2} - 2k^2 - k} \leq 2^{\binom{n}{2} - 3k} \leq 2^{\binom{n}{2}} / 8$ tournaments contain a $k$-triple, as $n^{3k} \leq 2^{2k^2 - 2k}$.
  Also, there are at most~$3^n$ pairs of disjoint sets~$X, Y \subseteq V$, and for each such pair with $|X| \cdot |Y| \geq 3n$, exactly $2^{\binom{n}{2} - |X| \cdot |Y|} \leq 2^{\binom{n}{2} - 3n}$ tournaments have no arc from~$X$ to~$Y$.
  Thus at most ${(3/8)}^n \cdot 2^{\binom{n}{2}} < 2^{\binom{n}{2}} / 2$ tournaments have such a pair with no arc from~$X$ to~$Y$.
  Therefore, some tournament~$T$ has no $k$-triple and an arc from~$X$ to~$Y$ for all disjoint $X, Y \subseteq V$ with $|X| \cdot |Y| \geq 3n$.
  Let $w = \dpw(T)$, and assume that $w < n - 1$, as otherwise there is nothing to show.
  As explained in \zcref{sec:prelim},~$T$ has a nice path decomposition of width~$w$, so \zcref{prop:separations-from-nice-path-decomposition} gives an index~$i$ with $|A_i| = w + \lceil (n - w)/2 \rceil$ and $|A_i \cap B_i| \leq w$.
  Then $|A_i \setminus B_i| \geq \lceil (n - w)/2 \rceil$, $|B_i \setminus A_i| = n - |A_i| = \lfloor (n - w)/2 \rfloor$, and there is no arc from $A_i \setminus B_i$ to $B_i \setminus A_i$.
  By the choice of~$T$, we obtain $(n - w)(n - w - 1)/4 \leq |A_i \setminus B_i| \cdot |B_i \setminus A_i| < 3n$, so ${(n - w - 1)}^2 < 12n$ and $w > n - 2\sqrt{3n} - 1$.

  For the \hsemi{h} digraph, let~$D$ be the digraph on disjoint sets~$X$ and~$Y$ of size~$h + 1$ with both arcs between every vertex of~$X$ and every vertex of~$Y$, and no other arcs.
  Every vertex of~$D$ is non-adjacent exactly to the other~$h$ vertices of its side, so~$D$ is \hsemi{h}.
  Every $k$-triple~$(A, B, C)$ of~$D$ gives vertices~$a \in A$, $b \in B$, and $c \in C$ with $(a, b), (b, c), (c, a) \in E(D)$.
  These three vertices are pairwise adjacent, which is impossible as two of them lie on the same side.
  We conclude that~$D$ has no $k$-triple.
  For all distinct $x, x' \in X$, the paths~$(x, y, x')$ with $y \in Y$ are internally disjoint, so~$X$ is a jungle of size~$h + 1$.
  Thus $\dpw(D) \geq h$ by \zcref{lem:jungle-dpw}, and the result follows.
\end{proof}

Since $n - 2\sqrt{3n} - 1 = 2^{(2/3 - o(1))k}$ and tournaments are \hsemi{h} for every~$h$, the tournament and the \hsemi{h} digraph of \zcref{lem:large-dpw-to-triples-tight} give the lower bounds~$2^{(2/3 - o(1))k}$ and~$h$, respectively.

\raggedbottom
\section{Arc-disjoint paths in \titlemath{h}-semicomplete digraphs}
\label{sec:adp-algorithm}
In this section, we give a fixed-parameter algorithm for \ADP on \hsemi{h} digraphs parameterized by $k + h$, where~$k$ is the number of paths (\zcref{thm:adp-fpt-simplified}).
We follow the algorithm of \textcite[Theorem 9.9]{FominP19} for semicomplete digraphs, which deletes a vertex of a large triple without changing the answer~\mbox{\cite[Theorem 9.1]{FominP19}} until it finds a path decomposition of small width for its dynamic programming~\cite[Theorem 10.8]{FominP19}.
Here the path decomposition must cover the non-arcs, which leads to a new obstruction, the anti-triples of \zcref{sec:anti-triples}, and in \zcref{sec:adp-detours-irrelevant-vertices} we add an arc between two non-neighbors of a large anti-triple, which does not change the answer by a result of \textcite[(2.4)]{FradkinS15}.

\subsection{Dynamic programming over path decompositions covering the non-arcs}\label{sec:adp-dynamic-programming}

We first fix the terminology of the section, and then we solve \ADP when a path decomposition covering the non-arcs is given.

\begin{boxproblem}{Arc-Disjoint Paths ({\sc ADP})}
  \label{problem:ADP}%
  Input: & A digraph~$D$, a positive integer~$k$, and a multiset~$\mathcal{R}$ of~$k$ pairs~$(s_1, t_1), \ldots, (s_k, t_k)$ of vertices of~$D$.\\
  Goal: & Decide whether there are pairwise arc-disjoint paths~$P_1, \ldots, P_k$ in~$D$ such that~$P_i$ goes from~$s_i$ to~$t_i$ for every $i \in [k]$.
\end{boxproblem}
\noindent We call the pairs~$(s_i, t_i)$ the \emph{requests}, the vertices~$s_i$ and~$t_i$ the \emph{terminals}, and we denote by $V(\mathcal{R})$ the set of terminals.
A \emph{solution} of $(D, k, \mathcal{R})$ is a sequence~$(P_1, \ldots, P_k)$ of paths as in the goal of \ADP, and $(D, k, \mathcal{R})$ is a \emph{``yes''-instance} if it has a solution.
Two instances are \emph{equivalent} if both or neither of them are yes-instances, and a vertex~$x$ of~$D$ is \emph{irrelevant} for $(D, k, \mathcal{R})$ if $(D \setminus \{x\}, k, \mathcal{R})$ and $(D, k, \mathcal{R})$ are equivalent.
A path \emph{uses} its vertices and its arcs, a solution uses the vertices and arcs of its paths, and the arcs that a solution does not use are \emph{unused}.
Since a path visits each vertex at most once, a solution uses at most~$k$ arcs leaving any vertex and at most~$k$ arcs entering it~\cite[Observation 9.2]{FominP19}.
Finally, every walk~$(p_0, \ldots, p_r)$ \emph{contains} a path from~$p_0$ to~$p_r$ whose arcs are arcs of the walk, obtained by deleting $p_{i+1}, \ldots, p_j$ while $p_i = p_j$ for some $i < j$.

\flushbottom
An \emph{immersion} of a digraph~$H$ into a digraph~$D$ maps the vertices of~$H$ to distinct vertices of~$D$ and every arc~$(y, z)$ of~$H$ to a path of~$D$ from the image of~$y$ to the image of~$z$, such that these paths are pairwise arc-disjoint~\cite[Section 2.3]{FominP19}.
\textsc{Rooted Immersion}, where the images of some vertices of~$H$ are prescribed, is solved on semicomplete digraphs by dynamic programming over a path decomposition~\mbox{\cite[Theorem~10.8]{FominP19}}.
This proof makes the path decomposition nice and uses semicompleteness only when a bag~$X_{i+1} = X_i \cup \{v\}$ introduces a vertex~$v$, where it needs that every vertex forgotten by the bags~$X_1, \ldots, X_i$ is an out-neighbor of~$v$~\cite[proof of Theorem 10.8, Cases 1.2, 1.4 and 2]{FominP19}.
By \zcref{obs:covering-non-arcs}, every nice path decomposition that covers the non-arcs has this property, so their proof applies without change and gives the following.

\begin{proposition}[Adapted from {\textcite[Theorem 10.8]{FominP19}}]\label{prop:fp-immersion-dp}
  There is an algorithm that, given a digraph~$H$, a digraph~$D$ on~$n$ vertices with a nice path decomposition of width~$p$ that covers the non-arcs of~$D$, and pairwise distinct prescribed images in $V(D)$ for some of the vertices of~$H$, decides in $2^{\Ocal(|V(H)| \cdot \log p + p \cdot |E(H)| \cdot (\log |E(H)| + \log p))} \cdot n$ time whether~$H$ has an immersion into~$D$ that respects the prescribed images.
\end{proposition}

We reduce \ADP to this problem.

\begin{lemma}\label{lem:adp-dp}
  There is an algorithm that, given an instance~$(D, k, \mathcal{R})$ of {\sc ADP} with $n = |V(D)|$ and a path decomposition of~$D$ of width~$w$ with~$r$ bags that covers the non-arcs of~$D$, in time~$2^{\Ocal(k(w+k) \cdot \log(w+k+1))} \cdot n + \Ocal((w + k + 1) \cdot r)$ decides if $(D, k, \mathcal{R})$ is a ``yes''-instance.
\end{lemma}

\begin{proof}
  First, we modify the instance so that no two requests are equal and no request has equal terminals, since the digraph~$H$ below must have neither loops nor parallel arcs.
  Let~$\mathcal{X}$ be the given path decomposition.
  For every $i \in [k]$, we add to~$D$ a new vertex~$s'_i$ with the single arc~$(s'_i, s_i)$, we add~$s'_i$ to every bag of~$\mathcal{X}$, and we replace the request~$(s_i, t_i)$ by $(s'_i, t_i)$.
  Let~$D_0$,~$\mathcal{X}_0$, and~$\mathcal{R}_0$ be the results.
  Since $(s'_i, s_i)$ is the only arc at~$s'_i$, a path from~$s'_i$ to~$t_i$ in~$D_0$ consists of this arc followed by a path from~$s_i$ to~$t_i$ in~$D$.
  The arcs~$(s'_1, s_1), \ldots, (s'_k, s_k)$ are distinct, so $(D_0, k, \mathcal{R}_0)$ is equivalent to $(D, k, \mathcal{R})$.
  The sequence~$\mathcal{X}_0$ is a path decomposition of~$D_0$ of width at most $w + k$ with~$r$ bags.
  It covers the non-arcs of~$D_0$, since every non-arc of~$D_0$ is a non-arc of~$D$ or has an endpoint~$s'_i$, which belongs to every bag.

  Let~$H$ be the digraph with vertex set~$V(\mathcal{R}_0)$ and arc set~$\{(s'_i, t_i) \mid i \in [k]\}$.
  These arcs have distinct tails and are not loops, so~$H$ has neither loops nor parallel arcs, and $|V(H)| \leq 2k$.
  By the definition of an immersion, $(D_0, k, \mathcal{R}_0)$ is a yes-instance if and only if there is an immersion of~$H$ into~$D_0$ mapping every vertex of~$H$ to itself.
  We make~$\mathcal{X}_0$ nice in $\Ocal((w + k + 1) \cdot (n + k + r))$ time, which preserves the covering of the non-arcs as noted after \zcref{def:covering-non-arcs}.
  Then we run the algorithm of \zcref{prop:fp-immersion-dp} on~$H$,~$D_0$, this nice path decomposition, and the prescribed image~$y$ of every vertex~$y$ of~$H$.
  Since $|V(H)| \leq 2k$, $|E(H)| = k$, the width is at most $w + k$, and~$D_0$ has $n + k$ vertices, its run time is $2^{\Ocal(k(w+k) \cdot \log(w+k+1))} \cdot n$, and the result follows.
\end{proof}

\subsection{Anti-triples}
\label{sec:anti-triples}

Anti-triples, together with triples, are the obstructions to path decompositions of small width that cover the non-arcs of an \hsemi{h} digraph (\zcref{lem:covering-obstructions}).

For a digraph~$D$, we denote by~$\hat{D}$ the completion of~$D$ (\zcref{def:detours-and-completions}) obtained by adding both arcs~$(u, v)$ and $(v, u)$ for every non-arc~$(u, v)$ of~$D$.
A path decomposition of~$D$ covers the non-arcs of~$D$ if and only if it is a path decomposition of~$\hat{D}$.
To see this, note that the two arcs of~$\hat{D}$ between non-neighbors~$u$ and~$v$ of~$D$ satisfy \caseref{def:pathwidth}{def:pathwidth-item-2} exactly when some bag contains~$u$ and~$v$.
Thus, by \zcref{thm:large-dpw-to-triples} and \zcref{cor:triple-dpw} applied to~$\hat{D}$, the triples of~$\hat{D}$ are the obstructions to path decompositions of small width that cover the non-arcs of~$D$.
We find two kinds of them in~$D$, namely the triples of~$D$ and the following new obstruction.

\begin{definition}[$k$-anti-triple]\label{def:anti-triple}
  Let~$D$ be a digraph and~$k$ be a positive integer.
  A \emph{$k$-anti-triple} of~$D$ is a triple~$(A, B, C)$ of pairwise disjoint subsets of $V(D)$ of size~$k$ such that~$A$ is complete to~$B$,~$B$ is complete to~$C$, and there are orderings~$(a_1, \ldots, a_k)$ of~$A$ and $(c_1, \ldots, c_k)$ of~$C$ such that~$a_i$ and~$c_i$ are non-neighbors for every $i \in [k]$.
\end{definition}

Every $k$-triple of~$D$ is a $k$-triple of~$\hat{D}$, as $E(D) \subseteq E(\hat{D})$, and so is every $k$-anti-triple of~$D$, as~$\hat{D}$ has the arc~$(c_i, a_i)$ for every $i \in [k]$.
Thus, by \zcref{cor:triple-dpw}, a $k$-anti-triple of~$D$ forces every path decomposition of~$D$ that covers the non-arcs of~$D$ to have width at least $k - 1$, even if~$D$ is acyclic.
If $(A, B, C)$ is a $k$-anti-triple of~$D$ and $a \in A$ and $c \in C$ are non-neighbors, then $B \subseteq \detours{D}{a}{c}$, since~$A$ is complete to~$B$ and~$B$ is complete to~$C$.
See \zcref{fig:adp-anti-triple}.

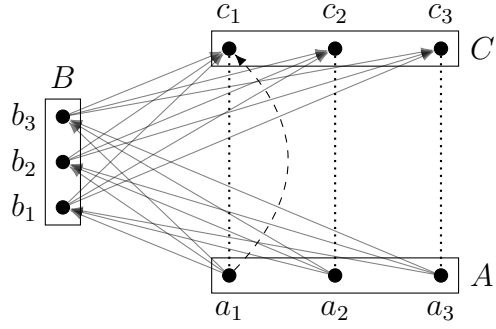
\begin{figure}[ht]
  \centering
  \begin{tikzpicture}[yscale=.6]
  \foreach \i in {1,...,3} {
      \node[blackvertex, label={[yshift=-.1cm]-90:{$a_\i$}}] (a\i) at (1.4*\i-1.4,-.5) {};
      \node[blackvertex, label={[xshift=-.1cm]180:{$b_\i$}}] (b\i) at (-2.2,\i) {};
      \node[blackvertex, label={[yshift=.1cm]90:{$c_\i$}}] (c\i) at ($(a\i) + (0,5)$) {};
      \draw[dotted, thick] (c\i) -- (a\i);
    }
  \foreach \i in {1,...,3}
  \foreach \j in {1,...,3} {
      \draw[arrow,opacity=.5] (a\i) -- (b\j);
      \draw[arrow,opacity=.5] (b\i) -- (c\j);
    }
  \draw[arrow, dashed] (a1) to[bend right=30] (c1);
  \node[rectangle,draw,fit=(a1)(a3), label=0:{$A$}] {};
  \node[rectangle,draw,fit=(b1)(b3), label=90:{$B$}] {};
  \node[rectangle,draw,fit=(c1)(c3), label=0:{$C$}] {};
\end{tikzpicture}%
%
  \caption{A $3$-anti-triple~$(A, B, C)$. Dotted segments join the non-neighbors~$a_i$ and~$c_i$. Every vertex of~$B$ is the internal vertex of a detour from~$a_1$ to~$c_1$, so $B \subseteq \detours{D}{a_1}{c_1}$. The dashed arc~$(a_1, c_1)$ is the arc that the algorithm of \zcref{thm:adp-fpt-simplified} adds.}
  \label{fig:adp-anti-triple}
\end{figure}

We now find the path decomposition or one of the two obstructions.
We run the algorithm of \zcref{prop:semicomplete-jungle-algorithm} on~$\hat{D}$ rather than on the detour completion~$D^\star$ of~$D$, since a path decomposition of~$D^\star$ need not cover the non-arcs of~$D$.
Then \zcref{lem:jungle-extraction} turns the short jungle into a triple of~$D$ or into pairs of non-neighbors that give an anti-triple of~$D$.

\begin{lemma}\label{lem:covering-obstructions}
  Let~$\ell$ be a positive integer,~$h$ be a non-negative integer, $J = \fref{lem:jungle-extraction}(\ell, h, \ell)$, and $\fref{lem:covering-obstructions}(\ell, h) = 140J$.
  There is an algorithm with run time~$\Ocal(J^3 \cdot n^2)$ that, given an \hsemi{h} digraph~$D$ on~$n$ vertices, outputs a path decomposition of~$D$ of width at most $\fref{lem:covering-obstructions}(\ell, h)$ that covers the non-arcs of~$D$, an $\ell$-triple of~$D$, or an $\ell$-anti-triple of~$D$.
\end{lemma}

\begin{proof}
  We compute~$\hat{D}$ in $\Ocal(n^2)$ time and run the algorithm of \zcref{prop:semicomplete-jungle-algorithm} on~$\hat{D}$ with $K = J$, which takes $\Ocal(J^3 \cdot n^2)$ time.
  If it outputs a path decomposition~$\mathcal{X}$ of~$\hat{D}$ of width at most $140J$, then~$\mathcal{X}$ is a path decomposition of~$D$ that covers the non-arcs of~$D$, and we output it.
  Otherwise, it outputs a $(J, 4)$-short jungle~$Z$ of~$\hat{D}$.
  Since~$\hat{D}$ is a completion of~$D$, we apply the algorithm of \zcref{lem:jungle-extraction} to~$D$, $S = \hat{D}$, and~$Z$, with~$\ell$ in place of~$k$ and of~$s$.
  This takes $\Ocal(\ell^2 \cdot J^3)$ time, which is $\Ocal(J^3 \cdot n^2)$ as $\ell \leq J \leq |Z| \leq n$.
  If it outputs an $\ell$-triple of~$D$, we output it.
  Otherwise, it outputs a set~$B$ of size~$2\ell$ and pairs~$(a_1, c_1), \ldots, (a_\ell, c_\ell)$ of non-neighbors of~$D$ as in its second item.
  Let $A = \{a_1, \ldots, a_\ell\}$, $C = \{c_1, \ldots, c_\ell\}$, and $B' \subseteq B$ be a set of size~$\ell$.
  Since the $2\ell$ vertices of the pairs are distinct and outside~$B$, the second item shows that $(A, B', C)$ is an $\ell$-anti-triple of~$D$, and we output it.
\end{proof}

\subsection{Detours and irrelevant vertices}\label{sec:adp-detours-irrelevant-vertices}

We now justify that the algorithm can add an arc between two non-neighbors of a large anti-triple and delete a vertex of a large triple without changing the answer.
We begin with the arc.
For a non-arc~$(u, v)$ of~$D$, we denote by $D + (u, v)$ the digraph obtained from~$D$ by adding the arc~$(u, v)$.
\textcite{FradkinS15} consider digraphs~$G$ and~$H$ with the same number of prescribed vertices, called \emph{roots}, where the roots of~$H$ are distinct.
For all distinct vertices~$u$ and~$v$ of~$G$, they show that adding arcs from~$u$ to~$v$ does not change whether~$H$ has an \emph{infusion} in~$G$ if there are at least $2|E(H)|$ detours from~$u$ to~$v$~\cite[(2.4)]{FradkinS15}.
When every vertex of~$H$ is a root and the ends of every arc of~$H$ correspond to distinct roots of~$G$, such an infusion maps every arc~$(x, y)$ of~$H$ to a path of~$G$ from the root corresponding to~$x$ to the root corresponding to~$y$, and these paths are pairwise arc-disjoint.
A request~$(s, s)$ is met by the path~$(s)$ in~$D$ and in $D + (u, v)$.
For the $k' \leq k$ other requests~$(s_i, t_i)$, we let~$H$ consist of~$k'$ pairwise disjoint arcs~$(x_i, y_i)$ whose ends are roots corresponding to~$s_i$ and~$t_i$.
Then the infusions of~$H$ in~$D$ and in $D + (u, v)$ are the solutions for these requests, and $2|E(H)| = 2k' \leq 2k$, so we obtain the following.

\begin{proposition}[{\textcite[(2.4)]{FradkinS15}}]\label{prop:avoidable-arcs}
  Let $(D, k, \mathcal{R})$ be an instance of {\sc ADP} and $u, v$ be non-neighbors of~$D$ with $|\detours{D}{u}{v}| \geq 2k$.
  Then $(D, k, \mathcal{R})$ and $(D + (u, v), k, \mathcal{R})$ are equivalent.
\end{proposition}

We now turn to triples.
\textcite[Theorem 9.1]{FominP19} show that, for \textsc{Rooted Immersion} in a semicomplete digraph, a large triple~$(A, B, C)$ avoiding the terminals contains an irrelevant vertex~$x \in B$.
We prove the analogue for \ADP in \hsemi{h} digraphs in \zcref{lem:adp-irrelevant-vertex}.
As in their proof, a path through~$x$ is rerouted through~$B$ and an arc from~$C$ to~$A$ of the triple.
Their choice of~$x$, however, needs semicompleteness for a lemma on bipartite digraphs with an arc between any two vertices on opposite sides~\cite[Lemma 9.8]{FominP19}, and we choose~$x$ with the following counting lemma, which only needs that the input is \hsemi{h}.
For a vertex~$u$ with few out-neighbors in~$B$, the set~$\Lambda(u)$ below consists of the vertices of~$B$ that~$u$ reaches by an arc or through another such vertex.

\begin{lemma}\label{lem:rich-vertices}
  Let~$D$ be an \hsemi{h} digraph on~$n$ vertices, $B \subseteq V(D)$, and~$q$ and~$d$ be positive integers.
  Let~$W$ be the set of vertices of~$D$ with at most~$q$ out-neighbors in~$B$.
  For $u \in W$, let $\Lambda(u)$ be the set of vertices~$b \in B$ with $(u, b) \in E(D)$, or with $(u, u') \in E(D)$ and $(u', b) \in E(D)$ for some $u' \in W$, and call~$u$ \emph{poor} if $|\Lambda(u)| < d$.
  Then the set~$Y$ of vertices of~$B$ with a poor in-neighbor in~$W$ has size at most $q(2d + h - 1)$ and can be computed in $\Ocal(q \cdot n^2)$ time.
\end{lemma}

\begin{proof}
  If $Y = \emptyset$, there is nothing to prove.
  Otherwise, for every $b \in Y$, choose a poor in-neighbor~$w_b \in W$ of~$b$, and let $W_0 = \{w_b \mid b \in Y\}$.
  Every $w \in W_0$ has at most~$q$ out-neighbors in~$B$, and~$b$ is an out-neighbor of~$w_b$, so $|Y| \leq q \cdot |W_0|$.
  Since $D[W_0]$ is \hsemi{h}, by \zcref{obs:out-degree-h-semicomplete} some $v \in W_0$ has at least $(|W_0| - 1 - h)/2$ out-neighbors in~$W_0$.
  Each out-neighbor~$w$ of~$v$ in~$W_0$ is~$w_b$ for some $b \in Y$, and then $b \in \Lambda(v)$, as $w \in W$ and $(v, w), (w, b) \in E(D)$.
  Distinct out-neighbors~$w$ give distinct vertices~$b$, since $w_b = w$, so $|\Lambda(v)| \geq (|W_0| - 1 - h)/2$.
  As~$v$ is poor, $|\Lambda(v)| \leq d - 1$, and we conclude that $|W_0| \leq 2d + h - 1$ and $|Y| \leq q(2d + h - 1)$.

  For the run time, we compute $N^+(u) \cap B$ for every $u \in V(D)$ in $\Ocal(n^2)$ time, which gives~$W$.
  For $u \in W$, the set~$\Lambda(u)$ is the union of $N^+(u) \cap B$ and of the sets~$N^+(u') \cap B$ over the out-neighbors~$u' \in W$ of~$u$.
  These sets have size at most~$q$, so the sets~$\Lambda(u)$ for all $u \in W$ are computed in $\Ocal(q \cdot n^2)$ time.
  Finally,~$Y$ is the union of the sets~$N^+(u) \cap B$ over the poor vertices~$u \in W$, and the result follows.
\end{proof}

Next, a solution with the minimum total number of arcs uses few vertices of a large triple.
\textcite[Observations 9.3 to 9.6]{FominP19} prove this for \textsc{Rooted Immersion} in semicomplete digraphs, with bounds in terms of the number~$|H|$ of vertices and arcs of the immersed digraph~$H$, for the large triple of their irrelevant vertex rule~\cite[Theorem 9.1]{FominP19}.
Their proofs never use semicompleteness or the roots, and they use the size of the triple only through the fact that it exceeds $4|H|$.
Otherwise, they use $|H|$ only as an upper bound on the number of paths, which bounds the used arcs at a vertex~\cite[Observation 9.2]{FominP19} and the used arcs of the matching.
Each proof replaces a subpath of a path by a shorter walk with the same ends whose arcs are unused arcs of the triple, and the path contained in this walk gives a solution with fewer arcs in total.
The bound for~$C$, which they obtain by symmetry, follows from the bound for~$A$ after reversing all arcs.
Thus their proofs, with the number~$k$ of paths in place of $|H|$, give the following for \ADP in every digraph.

\begin{proposition}[Adapted from {\textcite[Observations 9.3 to 9.6]{FominP19}}]\label{prop:solution-meets-triple}
  Let $(D, k, \mathcal{R})$ be an instance of {\sc ADP}, $(A, B, C)$ be an $N$-triple of~$D$ with $N > 4k$, and~$M$ be a set of~$N$ arcs from~$C$ to~$A$ witnessing that~$C$ is matched to~$A$.
  Let~$\mathcal{Q}$ be a solution of $(D, k, \mathcal{R})$ with the minimum total number of arcs.
  Then every path of~$\mathcal{Q}$ uses at most two arcs of~$M$, at most $2k + 4$ vertices of each of~$A$ and~$C$, and at most four vertices of~$B$.
\end{proposition}

We are ready to find the irrelevant vertex.

\begin{lemma}[Adapted from {\textcite[Theorem 9.1]{FominP19}}]\label{lem:adp-irrelevant-vertex}
  Let~$k$ be a positive integer,~$h$ be a non-negative integer, and $\fref{lem:adp-irrelevant-vertex}(k, h) = 20k(20k^2 + 20k + h + 1) + 1$.
  Let $(D, k, \mathcal{R})$ be an instance of {\sc ADP} where~$D$ is an \hsemi{h} digraph on~$n$ vertices, and $(A, B, C)$ be an $\fref{lem:adp-irrelevant-vertex}(k, h)$-triple of~$D$ with $(A \cup B \cup C) \cap V(\mathcal{R}) = \emptyset$.
  There is an algorithm with run time~$\Ocal(k \cdot n^2)$ that finds an irrelevant vertex~$x \in B$ for $(D, k, \mathcal{R})$.
\end{lemma}

\begin{figure}[ht]
  \centering
  \begin{tikzpicture}
  \node[blackvertex, label={[xshift=-.1cm]180:{$v$}}] (v) at (-4, 0) {};
  \node[blackvertex, label={[xshift=.1cm]0:{$w$}}] (w) at (4, 0) {};
  \node[blackvertex, label={[yshift=.1cm]90:{$u$}}] (u) at (-2.6, 1.7) {};
  \node[blackvertex, label={[yshift=.1cm]90:{$u'$}}] (up) at (2.6, 1.7) {};
  \node[blackvertex, label={[xshift=.1cm, yshift=.1cm]45:{$x$}}] (x) at (0, 0) {};
  \node[blackvertex, label={[xshift=-.1cm]180:{$b_2$}}] (b2) at (0, 1.2) {};
  \node[blackvertex, label={[xshift=.1cm]0:{$b_1$}}] (b1) at (0, 2.4) {};
  \node[blackvertex, label={[yshift=.1cm]90:{$c$}}] (c) at (-1.5, 3.9) {};
  \node[blackvertex] (c2) at (-2.3, 3.9) {};
  \node[blackvertex, label={[yshift=.1cm]90:{$a$}}] (a) at (1.5, 3.9) {};
  \node[blackvertex] (a2) at (2.3, 3.9) {};
  \draw[arrow, dashed] (v) -- (x);
  \draw[arrow, dashed] (x) -- (w);
  \draw[arrow] (v) -- (u);
  \draw[arrow] (u) -- (b1);
  \draw[arrow] (b1) -- (c);
  \draw[arrow] (c) -- (a);
  \draw[arrow] (a) -- (b2);
  \draw[arrow] (b2) -- (up);
  \draw[arrow] (up) -- (w);
  \draw[arrow, opacity=.5] (c2) to[bend left=20] (a2);
  \node[rectangle, draw, inner sep=5pt, fit=(b1)(b2)(x), label=-90:{$B$}] {};
  \node[rectangle, draw, inner sep=5pt, fit=(c)(c2), label=180:{$C$}] {};
  \node[rectangle, draw, inner sep=5pt, fit=(a)(a2), label=0:{$A$}] {};
\end{tikzpicture}%
%
  \caption{Rerouting a path through~$x$ in the proof of \zcref{lem:adp-irrelevant-vertex}. The subpath~$(v, x, w)$, drawn dashed, is replaced by the walk~$(v, u, b_1, c, a, b_2, u', w)$, where $(c, a)$ is an arc of the matching. The gray arc is another arc of the matching, and the arcs from~$A$ to~$B$ and from~$B$ to~$C$ other than $(a, b_2)$ and $(b_1, c)$ are omitted.}
  \label{fig:adp-irrelevant-vertex}
\end{figure}
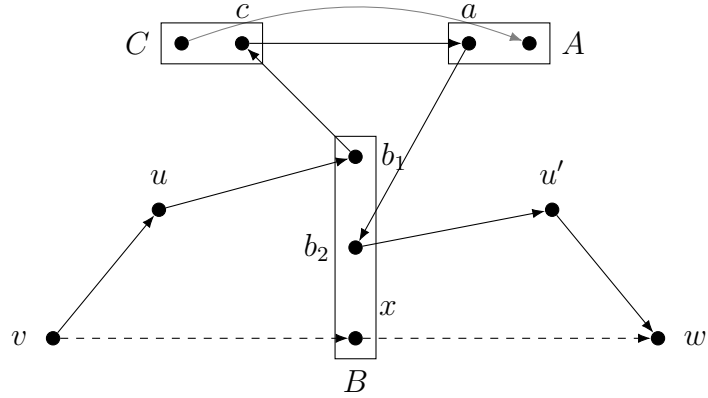

\begin{proof}
  Choosing $N = \fref{lem:adp-irrelevant-vertex}(k, h)$, $q = 10k$, and $d = 10k^2 + 10k + 1$, we have \mbox{$N = 2q(2d + h - 1) + 1$}.
  We fix a set of~$N$ arcs from~$C$ to~$A$ witnessing that~$C$ is matched to~$A$, which we call the \emph{matching}.

  First, we choose~$x$.
  Let~$W$ and~$Y$ be given by \zcref{lem:rich-vertices} applied to~$D$,~$B$,~$q$, and~$d$.
  Let~$Y'$ be obtained in the same way from the digraph obtained from~$D$ by reversing all arcs, which is \hsemi{h}, so that $|Y'| \leq q(2d + h - 1)$.
  Thus $|Y| + |Y'| < N = |B|$, and we choose $x \in B \setminus (Y \cup Y')$.

  Every solution of $(D \setminus \{x\}, k, \mathcal{R})$ is a solution of $(D, k, \mathcal{R})$.
  Conversely, let~$\mathcal{Q}$ be a solution of $(D, k, \mathcal{R})$ with the minimum total number of arcs, and let~$U$ be the set of vertices used by~$\mathcal{Q}$.

  Since~$\mathcal{Q}$ has~$k$ paths and $N > 4k$, \zcref{prop:solution-meets-triple} applied to the matching gives $|U \cap B| \leq 4k$ and $|U \cap A|, |U \cap C| \leq k(2k + 4)$.
  Thus at most $2k(2k + 4)$ arcs of the matching have an endpoint in~$U$, as every vertex of $A \cup C$ is an endpoint of exactly one of them.

  If no path of~$\mathcal{Q}$ uses~$x$, we are done.
  Otherwise $x \in U$, and we reroute the paths through~$x$ one at a time, at most~$k$ times.
  We call a vertex \emph{taken} if it belongs to~$U$ or was chosen in a previous rerouting, so that every arc with an endpoint that is not taken is unused by the current solution.
  Each rerouting chooses at most six vertices.
  Thus at most $4k + 6(k - 1) < 10k$ vertices of~$B$ are taken during a rerouting, and at most $2k(2k + 4) + 6(k - 1) < N$ arcs of the matching have a taken endpoint.

  Let~$P$ be a path of the current solution that uses~$x$, and let~$v$ and~$w$ be the vertices before and after~$x$ on~$P$, which exist since~$x$ is not a terminal.
  We replace the subpath~$(v, x, w)$ of~$P$ by a walk from~$v$ to~$w$ that avoids~$x$ and whose arcs are unused.
  The walk consists of a walk~$R_1$ from~$v$ to a vertex~$b_1 \in B$, the path~$(b_1, c, a, b_2)$, and a walk~$R_2$ from~$b_2$ to~$w$.
  Here $(c, a)$ is an arc of the matching, and~$b_1$,~$c$,~$a$, and~$b_2$ are not taken.
  See \zcref{fig:adp-irrelevant-vertex}.
  The arc~$(c, a)$ exists by the bound on the arcs of the matching with a taken endpoint.
  The arcs~$(b_1, c)$ and $(a, b_2)$ exist since~$B$ is complete to~$C$ and~$A$ is complete to~$B$.

  We find~$R_1$ as follows.
  If~$v$ has more than~$q$ out-neighbors in~$B$, then one of them, say~$b_1$, is not taken, and we let $R_1 = (v, b_1)$.
  Otherwise $v \in W$, and~$v$ is not poor, as~$v$ is an in-neighbor of $x \notin Y$.
  Thus more than $d - 10k = 10k^2 + 1$ vertices~$b \in \Lambda(v)$ are not taken.
  For each of them, fix $u_b \in W$ with $(u_b, b) \in E(D)$ such that $u_b = v$ or $(v, u_b) \in E(D)$.
  If $u_b = v$ for one of them, we let $b_1 = b$ and $R_1 = (v, b_1)$.
  Otherwise, since every vertex of~$W$ has at most $q = 10k$ out-neighbors in~$B$, there are at least $k + 1$ distinct vertices~$u_b$, and at most~$k$ used arcs leave~$v$.
  Thus $(v, u_b)$ is unused for some~$b$, and we let $b_1 = b$, $u = u_b$, and $R_1 = (v, u, b_1)$, where $u \neq x$ since $(v, x)$ is used.

  Symmetrically, as $x \notin Y'$, there are a vertex~$b_2 \in B$ that is not taken and a walk~$R_2$ from~$b_2$ to~$w$.
  The walk~$R_2$ is $(b_2, w)$ or $(b_2, u', w)$, where $u' \neq x$ and $(u', w)$ is unused.

  The resulting walk avoids~$x$, which is taken, and its arcs are unused.
  Replacing $(v, x, w)$ on~$P$ by it and taking the path contained in the result gives a solution in which~$P$ no longer uses~$x$ and the other paths are unchanged.
  After at most~$k$ reroutings, no path uses~$x$, and $(D \setminus \{x\}, k, \mathcal{R})$ is a yes-instance.

  The run time follows from \zcref{lem:rich-vertices}, which computes the sets~$Y$ and~$Y'$ in $\Ocal(q \cdot n^2) = \Ocal(k \cdot n^2)$ time, and then~$x$ is found in $\Ocal(n)$ time.
\end{proof}

\subsection{The algorithm}\label{sec:adp-fpt-algorithm}

We now prove the main result of this section.
The algorithm follows the one for semicomplete digraphs~\cite[proof of Theorem 9.9]{FominP19}, with the anti-triple as a new third case of each round, as shown in \zcref{fig:adp-rounds}.

\getkeytheorem{adpfpt}

\begin{proof}
  Let $(D, k, \mathcal{R})$ be the given instance, $T = V(\mathcal{R})$, and $\ell = \fref{lem:adp-irrelevant-vertex}(k, h)$, so that $|T| \leq 2k$ and $\ell \geq 2k$.
  We maintain a digraph~$D'$ obtained from~$D$ by deleting vertices outside~$T$ and adding arcs between non-neighbors, such that $(D', k, \mathcal{R})$ is equivalent to $(D, k, \mathcal{R})$.
  Both operations preserve the property of being \hsemi{h}, so~$D'$ and $D' \setminus T$ are \hsemi{h}.
  We look for triples in $D' \setminus T$, so that they contain no terminal.
  Initially $D' = D$, and we repeat the following round.

  We first find a path decomposition~$\mathcal{X}$ of $D' \setminus T$ of width at most $\fref{lem:covering-obstructions}(\ell, h)$ that covers the non-arcs of $D' \setminus T$, an $\ell$-triple of $D' \setminus T$, or an $\ell$-anti-triple of $D' \setminus T$.
  If $D' \setminus T$ has no vertex, then $\mathcal{X} = (\emptyset)$ is such a path decomposition, and otherwise we run the algorithm of \zcref{lem:covering-obstructions} with~$\ell$ on $D' \setminus T$.
  In the first case, let~$\mathcal{X}'$ be obtained from~$\mathcal{X}$ by adding~$T$ to every bag.
  Since the vertices of~$T$ belong to every bag,~$\mathcal{X}'$ is a path decomposition of~$D'$ of width at most $\fref{lem:covering-obstructions}(\ell, h) + 2k$ that covers the non-arcs of~$D'$.
  We run the algorithm of \zcref{lem:adp-dp} on $(D', k, \mathcal{R})$ and~$\mathcal{X}'$, and we return its answer.
  In the second case, the $\ell$-triple of $D' \setminus T$ is an $\ell$-triple of~$D'$ that contains no terminal.
  Thus \zcref{lem:adp-irrelevant-vertex} gives an irrelevant vertex~$x$ for $(D', k, \mathcal{R})$, and we replace~$D'$ by $D' \setminus \{x\}$.
  In the third case, let $(A, B, C)$ be the $\ell$-anti-triple and $a \in A$ and $c \in C$ be non-neighbors.
  Then $B \subseteq \detours{D' \setminus T}{a}{c} \subseteq \detours{D'}{a}{c}$, as noted after \zcref{def:anti-triple}, so $|\detours{D'}{a}{c}| \geq \ell \geq 2k$, and $(D' + (a, c), k, \mathcal{R})$ is equivalent to $(D', k, \mathcal{R})$ by \zcref{prop:avoidable-arcs}.
  We replace~$D'$ by $D' + (a, c)$.

  Each round but the last deletes a vertex or makes two non-neighbors adjacent, and no round creates non-neighbors.
  As~$D$ has at most $hn/2$ pairs of non-neighbors by \zcref{obs:out-degree-h-semicomplete}, there are at most $(h + 2) \cdot n/2 + 1$ rounds.

  For the run time, let $J = \fref{lem:jungle-extraction}(\ell, h, \ell)$.
  Each round takes $\Ocal(J^3 \cdot n^2)$ time for \zcref{lem:covering-obstructions} and $\Ocal(k \cdot n^2)$ time for \zcref{lem:adp-irrelevant-vertex}.
  As $\ell = \Ocal(k^3 + kh)$ and $h + 1 \leq 2^{kh}$, we have $J = (h+1) \cdot 2^{\Ocal(\ell)} = 2^{\Ocal(k^3 + kh)}$.
  Thus the $\Ocal((h+1) \cdot n)$ rounds take $2^{\Ocal(k^3 + kh)} \cdot n^3$ time in total.
  The algorithm of \zcref{lem:adp-dp} runs once, on a path decomposition of width~$w \leq 140J + 2k = 2^{\Ocal(k^3 + kh)}$ with $2^{\Ocal(k^3 + kh)} \cdot n^2$ bags, since \zcref{lem:covering-obstructions} outputs it within this time.
  This takes $2^{s} \cdot n + 2^{\Ocal(k^3 + kh)} \cdot n^2$ time for some $s = 2^{\Ocal(k^3 + kh)}$, so the total run time is $2^{s} \cdot n^3$, and the result follows.
\end{proof}

\section{Cliquewidth of \titlemath{1}-semicomplete digraphs}\label{sec:cliquewidth-and-pathwidth}

In this section, we show that, unlike in semicomplete digraphs, directed pathwidth does not bound directed cliquewidth in \hsemi{1} digraphs, even acyclic ones (\zcref{thm:cliquewidth-and-pathwidth}).
\textcite[Lemma 2.14]{FominP19} show that the directed cliquewidth of a semicomplete digraph is at most its directed pathwidth plus two.
The following definition of directed cliquewidth goes back to \textcite{CourcelleO00}.

\begin{definition}[Directed cliquewidth]\label{def:directed-cliquewidth}
    A \emph{colored digraph} is a pair~$(D, \sigma)$ where~$D$ is a digraph and $\sigma \colon V(D) \to [p]$ for some positive integer~$p$.
    We call $\sigma(v)$ the \emph{color} of~$v$.
    A \emph{clique expression} is built from the constants~$\iota_a$, where~$\iota_a$ is the colored digraph with one vertex of color~$a$, using the following three operations on colored digraphs.
    \begin{outline}
        \1 The \emph{disjoint union} of $(D_1, \sigma_1)$ and $(D_2, \sigma_2)$ with $V(D_1) \cap V(D_2) = \emptyset$ is $(D_1 \cup D_2, \sigma_1 \cup \sigma_2)$.
        \1 For distinct colors~$i$ and~$j$, the \emph{directed join}~$\join_{i,j}$ of $(D, \sigma)$ adds to~$D$ every arc~$(u, v)$ with $\sigma(u) = i$ and $\sigma(v) = j$, and keeps~$\sigma$ unchanged.
        \1 For distinct colors~$i$ and~$j$, the \emph{recoloring}~$\mu_{i,j}$ of $(D, \sigma)$ gives color~$j$ to every vertex of color~$i$, and keeps~$D$ unchanged.
    \end{outline}
    A clique expression \emph{constructs} a digraph~$D$ if it evaluates to $(D, \sigma)$ for some coloring~$\sigma$.
    Its \emph{width} is the number of colors it uses, and the \emph{directed cliquewidth}~$\dcw(D)$ of~$D$ is the minimum width of a clique expression that constructs~$D$.
\end{definition}

We introduce the following digraph, which is acyclic and \hsemi{1} but has large directed cliquewidth, see \zcref{fig:cw-mirrored-grid} for~$S_4$.

\begin{definition}[Mirrored grid]\label{def:mirrored-grid}
    For a positive integer~$m$, the \emph{mirrored grid}~$S_m$ has the vertices~$v_{i,t}$ with $i, t \in [m]$.
    We call $C_i = \{v_{i,t} \mid t \in [m]\}$ the $i$-th \emph{column} of~$S_m$, and~$v_{i,t}$ is on the $t$-th \emph{row}.
    We order the vertices column by column and bottom-up, that is, $v_{i,t} \prec v_{j,s}$ if $i < j$, or if $i = j$ and $t < s$.
    The arcs of~$S_m$ are the pairs~$(v, w)$ with $v \prec w$, except that~$v_{i,t}$ and~$v_{t,i}$ are non-neighbors for all $i \neq t$.
\end{definition}

All arcs go forward in~$\prec$, so~$S_m$ is acyclic and $\dpw(S_m) = 0$.
Every vertex~$v_{i,t}$ has at most one non-neighbor, namely~$v_{t,i}$ when $i \neq t$, so~$S_m$ is \hsemi{1}, and the non-neighbors of the vertices of a column lie in distinct columns.

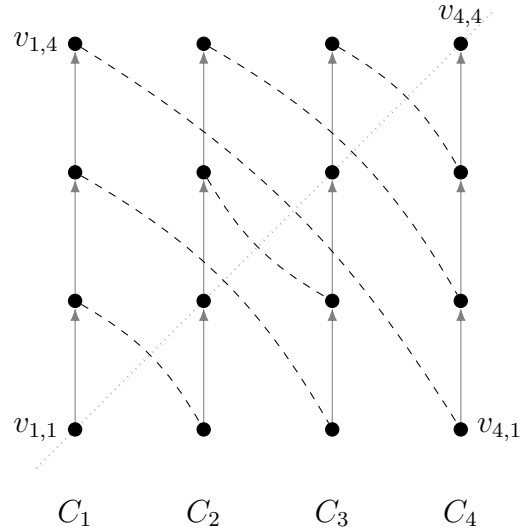
\begin{figure}[!htb]
    \centering
    \begin{tikzpicture}[scale=1.7]
  \draw[faded, dotted, thick] (.7, .7) -- (4.25, 4.25);
  \foreach \i in {1,...,4} {
      \node at (\i, .35) {$C_\i$};
      \foreach \t in {1,...,4} {
          \node[blackvertex] (v\i\t) at (\i, \t) {};
        }
      \foreach \t [evaluate=\t as \s using int(\t+1)] in {1,2,3} {
          \draw[arrow, opacity=.5] (v\i\t) -- (v\i\s);
        }
    }
  \foreach \i/\t/\b in {1/2/15, 1/3/15, 1/4/13, 2/3/-15, 2/4/15, 3/4/15} {
      \draw[dashed] (v\i\t) to[bend left=\b] (v\t\i);
    }
  \node[left=2pt] at (v11) {$v_{1,1}$};
  \node[left=2pt] at (v14) {$v_{1,4}$};
  \node[right=2pt] at (v41) {$v_{4,1}$};
  \node[above=2pt] at (v44) {$v_{4,4}$};
\end{tikzpicture}%
%
    \caption{The mirrored grid~$S_4$. Dashed curves join the pairs of non-neighbors~$v_{i,t}$ and~$v_{t,i}$, which are mirror images with respect to the dotted diagonal. Arcs omitted from the drawing go bottom-up inside each column and from left to right between columns.}
    \label{fig:cw-mirrored-grid}
\end{figure}

We are ready to prove \zcref{thm:cliquewidth-and-pathwidth}.

\getkeytheorem{cliquewidthpathwidth}

\begin{proof}
    Let $m = 6k$.
    Since~$S_m$ is acyclic and \hsemi{1}, it remains to prove that $\dcw(S_m) \geq k$.
    Let~$\varphi$ be a clique expression of width~$p$ that constructs~$S_m$.
    For a subexpression~$\psi$ of~$\varphi$, let $(D_\psi, \sigma_\psi)$ be the value of~$\psi$.
    A \emph{block} of $X \subseteq V(S_m)$ is a maximal interval of~$\prec$ contained in~$X$.
    The idea of the proof is the following.
    Two vertices of~$D_\psi$ with the same color have the same arcs to and from every vertex outside $V(D_\psi)$, since these arcs are created later in~$\varphi$.
    In~$S_m$, a vertex~$z \notin V(D_\psi)$ receives arcs from the vertices before it in~$\prec$ and sends arcs to those after it, with at most one exception.
    A vertex of~$D_\psi$ whose non-neighbor is outside $V(D_\psi)$ is the only vertex of~$D_\psi$ not adjacent to it.
    So~$p$ is large if $V(D_\psi)$ meets many columns, or contains much of one column but meets few columns (\zcref{clm:cw-colors}).
    One part of a minimal subexpression whose vertex set contains a column is of one of these kinds.

    \begin{claim}\label{clm:cw-colors}
        Let~$\psi$ be a subexpression of~$\varphi$ and $X = V(D_\psi)$.
        Then~$p$ is at least the number of blocks of~$X$ and at least the number of vertices of~$X$ with a non-neighbor outside~$X$.
    \end{claim}

    \begin{cproof}
        First, we show that two vertices~$x, y \in X$ with $\sigma_\psi(x) = \sigma_\psi(y)$ have the same in-neighbors and out-neighbors outside~$X$.
        Let $z \in V(S_m) \setminus X$.
        Every arc between~$z$ and~$X$ is created by a directed join applied to a subexpression of~$\varphi$ that has~$\psi$ as a subexpression.
        In each such subexpression,~$x$ and~$y$ have the same color, since recolorings act on whole color classes and the other operations do not change colors.
        Thus a directed join creates $(x, z)$ if and only if it creates $(y, z)$, and similarly for $(z, x)$ and $(z, y)$.

        Now, let $x, y \in X$ be in distinct blocks of~$X$ with $x \prec y$.
        By the maximality of blocks, there is $z \notin X$ with $x \prec z \prec y$.
        Since~$z$ has at most one non-neighbor, $(x, z)$ or $(z, y)$ is an arc.
        Neither $(y, z)$ nor $(z, x)$ is an arc, since all arcs go forward in~$\prec$.
        Thus~$x$ and~$y$ have distinct colors, and~$p$ is at least the number of blocks of~$X$.

        Finally, let $x \in X$ have a non-neighbor~$z \notin X$.
        Since~$x$ is the only non-neighbor of~$z$, every other vertex of~$X$ is adjacent to~$z$, so~$x$ is the only vertex of its color in~$D_\psi$, and the claim follows.
    \end{cproof}

    Let~$\psi$ be a minimal subexpression of~$\varphi$ such that $V(D_\psi)$ contains a column, say~$C_c$.
    By minimality,~$\psi$ is the disjoint union of two subexpressions whose vertex sets contain no column, as the other operations do not change the vertex set.
    For one of them, say~$\psi'$, the set~$X = V(D_{\psi'})$ contains at least $m/2 = 3k$ vertices of~$C_c$.
    Every block of~$X$ meets at most two columns, since a block that meets~$C_i$ and~$C_j$ with $j > i + 1$ contains~$C_{i+1}$, and~$X$ contains no column.
    If at least $2k$ columns meet~$X$, then~$X$ has at least~$k$ blocks, and $p \geq k$ by \zcref{clm:cw-colors}.
    Otherwise, at most $2k - 1$ columns meet~$X$, so among the at least $3k$ rows~$t$ with $v_{c,t} \in X$ there are at least $k + 1$ such that~$C_t$ does not meet~$X$.
    For each such~$t$, we have $t \neq c$ since~$C_c$ meets~$X$, so the non-neighbor~$v_{t,c}$ of~$v_{c,t}$ is not in~$X$, and $p \geq k + 1$ by \zcref{clm:cw-colors}.
    In both cases $p \geq k$, and the result follows.
\end{proof}

The following proposition shows that path decompositions covering the non-arcs bound the directed cliquewidth as in semicomplete digraphs.
The proof for semicomplete digraphs~\cite[Lemma 2.14]{FominP19} first makes the path decomposition nice, which preserves the covering of the non-arcs as noted after \zcref{def:covering-non-arcs}.
It then uses semicompleteness only when a bag introduces a vertex~$v$, where it needs that~$v$ is complete to the forgotten vertices; that no forgotten vertex has an arc to~$v$ holds in every digraph by \zcref{prop:separations-from-nice-path-decomposition}.
By \zcref{obs:covering-non-arcs}, every nice path decomposition that covers the non-arcs has this property, so this proof applies without change and gives the following.

\begin{proposition}[Adapted from {\textcite[Lemma 2.14]{FominP19}}]\label{prop:cw-covering-non-arcs}
    Let~$D$ be a digraph and~$\mathcal{X}$ be a path decomposition of~$D$ of width~$w$ that covers the non-arcs of~$D$.
    Then $\dcw(D) \leq w + 2$.
\end{proposition}

In particular, every path decomposition of~$S_{6k}$ that covers its non-arcs has width at least $k - 2$.

\printbibliography

\end{document}